\documentclass[12pt]{amsart}

\usepackage[usenames,dvipsnames]{xcolor}
\usepackage{enumitem,kantlipsum}

\usepackage{mathtools}
\usepackage{xfrac}
\usepackage[utf8]{inputenc}
\usepackage{amssymb}
\usepackage{amsmath}
\usepackage{amsfonts}
\usepackage{amsthm}
\usepackage{tikz}
\usepackage{algorithmic}
\usepackage{asymptote}
\usepackage{url}
\usepackage{hyperref}

\newtheorem{theorem}{Theorem}[section]
\newtheorem{corollary}{Corollary}[theorem]
\newtheorem{lemma}[theorem]{Lemma}
\newtheorem{proposition}[theorem]{Proposition}
\newtheorem{conjecture}[theorem]{Conjecture}
\newtheorem{definition}[theorem]{Definition}

\usepackage[OT2,T1]{fontenc}
\DeclareSymbolFont{cyrletters}{OT2}{wncyr}{m}{n}
\DeclareMathSymbol{\Sha}{\mathalpha}{cyrletters}{"58}

\theoremstyle{definition}

\newcommand{\PPP}{\mathbb{P}}
\newcommand{\PPb}[1]{\PPP\left[#1\right]}

\newcommand{\qpochn}[3]{\left(#1;#2\right)_{#3}}
\newcommand{\SSS}[3]{\mathcal{S}_{#1}(#2,#3)}
\newcommand{\SSO}[3]{\mathcal{S}^1_{#1}(#2,#3)}
\newcommand{\Mnot}{M_{n_1,n_2}}
\newcommand{\Lnot}{L_{n_1,n_2}}
\newcommand{\Mna}{M_{n}^{(\alpha)}}
\newcommand{\Lna}{L_{n}^{(\alpha)}}

\newcommand{\Erho}{\E_{\rho}}
\newcommand{\Erhop}{\E_{\rho}'}
\newcommand{\Vsigma}{V_{\setminus\sigma}}
\newcommand{\ZG}{\Z/\abs{G}\Z}

\newcommand{\abs}[1]{\vert#1\vert}
\newcommand{\Var}{\text{Var}}

\newcommand{\Sym}{\text{Sym}}
\newcommand{\msympinf}{P^{\Sym}_{\infty,p}}

\newcommand{\C}{\mathbb{C}}
\newcommand{\im}{\text{Im}}

\newcommand{\Znot}{Z_{n_1+n_2}}
\newcommand{\Zpn}{\Z_p^{n}}
\newcommand{\Zpnot}{\Z_p^{n_1+n_2}}
\newcommand{\Zp}{\Z_p}

\newcommand{\aln}{\lceil \alpha n\rceil}

\newcommand{\E}{\mathbb{E}}

\newcommand{\MM}{\mathbb{M}}
\newcommand{\Z}{\mathbb{Z}}

\newcommand{\Coker}{\text{Coker}}

\newcommand{\rank}{\text{rank}}

\newcommand{\Aut}{\text{Aut}}

\newcommand{\Hom}{\text{Hom}}
\newcommand{\Sur}{\text{Sur}}

\newcommand\restr[2]{{
  \left.\kern-\nulldelimiterspace 
  #1 
  \vphantom{\big|} 
  \right|_{#2} 
  }}

\newcommand{\ceil}[1]{\lceil#1\rceil}

\newtheorem{innercustomthm}{Theorem}
\newenvironment{customthm}[1]
  {\renewcommand\theinnercustomthm{#1}\innercustomthm}
  {\endinnercustomthm}

\newtheorem{innercustomprop}{Proposition}
\newenvironment{customprop}[1]
  {\renewcommand\theinnercustomprop{#1}\innercustomprop}
  {\endinnercustomprop}

\begin{document}
\title{Distribution of Sandpile groups of random directed bipartite graphs}
\author{Deepesh Singhal}

\begin{abstract}
Fix a prime $p$ and a constant $\frac{1}{p}<\alpha\leq 1$. Consider the random directed Erd\H{o}s--R\'enyi bipartite graph $\vec G(n,\ceil{\alpha n} ,v)$ with bipartition $(V_1,V_2)$ of sizes $|V_1|=n$ and $|V_2|=\ceil{\alpha n}$, and edge probability $0<v<1$. Bhargava, DePascale and Koenig \cite{bhargava2023rank} conjectured a limiting distribution for the $p$-Sylow subgroup of the sandpile group of $\vec G(n,\ceil{\alpha n},v)$ as $n\to\infty$.
We prove this conjecture.

Similar results have previously been proved by computing the expected number of surjections from the random abelian $p$-group onto $H$, for each finite abelian $p$-group $H$.
However, in the case of $p$-Sylow subgroups of sandpile groups of random directed bipartite graphs, these surjective moments often diverge to infinity, despite the conjectured limiting distribution having finite moments.
We resolve this by restricting to a high-probability subset of graphs on which the surjective moments are well-behaved, and discarding a rare exceptional set of graphs whose contribution to the distribution vanishes but whose contribution to the surjective moments often diverges. Computing the conditional surjective moments on the good set and applying Wood's universality theorem yields the desired convergence in distribution.
\end{abstract}

\maketitle

\section{Introduction}
Fix a prime $p$ throughout the paper. Define the $q$-shifted factorials
\[
\qpochn{x}{q}{i}
:= (1-x) (1-xq) \cdots (1-xq^{i-1}),
\quad
\qpochn{x}{q}{\infty} 
:= \prod_{j=0}^{\infty} (1-xq^j).
\]
Let $\vec\Gamma$ be a directed graph on vertices $1,\dots,n$ without any loops.  Write $\deg(i,j)$ for the number of directed edges $i\to j$, and set $\deg(j)=\sum_{i:i\neq j}\deg(i,j)$, the in-degree of the vertex.
Define the Laplacian $L(\vec\Gamma)$ by
\[
L_{ij}=\begin{cases}
\deg(i,j), & i\neq j,\\
-\deg(j), & i=j,
\end{cases}
\qquad\text{so that}\qquad
\sum_{i=1}^n L_{ij}=0\ \text{ for all }j.
\]
Let $\Z_0^n$ be the subset of $\Z^n$ consisting of vectors that sum up to $0$.
Hence $L:\Z^n\to\Z^n$ satisfies $\im(L)\subseteq \Z_0^n$.
We define the \emph{sandpile group} of $\vec{\Gamma}$ as $S_{\vec\Gamma}:=\Z_0^n/\im(L)$. 
If $\vec\Gamma$ is strongly connected, then $\text{corank}(L)=1$ and $S_{\vec\Gamma}$ is a finite abelian group.

Let $\vec G(n,v)$ denote the directed Erd\H{o}s-R\'enyi random graph on $n$ vertices in which each ordered pair $i\to j$ with $i\neq j$ is included independently with probability $v\in(0,1)$. For fixed $v$, $\vec G(n,v)$ is strongly connected with probability tending to $1$ as $n\to\infty$. Consequently, the sandpile group $S_{\vec G(n,v)}$ is a finite abelian group with probability tending to $1$.

For $0<u<p$, the measure $P_{\infty,u}$ on finite abelian $p$-groups is defined as
\[
P_{\infty,u}(G)
:=\frac{\abs{G}^{\log_p(u)}}{\abs{\Aut(G)}}
\qpochn{\tfrac{u}{p}}{\tfrac{1}{p}}{\infty}.
\]
Koplewitz \cite[Theorem 1]{koplewitz2017sandpile} shows that as $n \rightarrow \infty$, the distribution of the Sylow $p$-subgroups of the sandpile group of a directed Erd\H{o}s-R\'enyi random graph with $n$ vertices converges to $P_{\infty,\frac{1}{p}}$. 
\begin{theorem}\cite[Theorem 1]{koplewitz2017sandpile}
Consider a finite abelian $p$-group $G$. Then for a random directed graph~$\vec G(n,v)$,
\[
\lim_{n \rightarrow \infty} \PPb{(S_{\vec G(n,v)})_p \cong G} 
=P_{\infty,\frac{1}{p}}(G).
\]
\end{theorem}

Nguyen and Wood \cite[Theorem~1.6]{nguyen2022random} determined the distribution of the entire sandpile group of a directed Erd\H{o}s-R\'enyi random graph as $n\to\infty$.
\begin{theorem}\cite[Theorem~1.6]{nguyen2022random}
Consider a finite abelian group $G$. Then for a random directed graph~$\vec G(n,v)$,
\[
\lim_{n\to\infty} 
\PPb{S_{\vec G(n,v)}\cong G}
=
\frac{1}{\abs{G} \abs{\Aut(G)}}
\prod_{k=2}^{\infty}\zeta(k)^{-1}.
\]
\end{theorem}

We consider random directed bipartite graphs.  Let $n_1 \ge n_2$ be positive integers and fix $v\in (0,1)$. The \emph{directed Erd\H{o}s-R\'enyi random bipartite graph} $\vec G(n_1, n_2, v)$ is a directed bipartite graph with vertex sets $V_1$ of size $n_1$ and $V_2$ of size $n_2$, where each of the $2n_1 n_2$ potential directed edges from $V_1$ to $V_2$ and $V_2$ to $V_1$ is included independently with probability $v$. 
We are most interested in the case where $n_1$ and $n_2$ go to
infinity together and $n_2=\ceil{\alpha n_1}$ for a fixed constant $0<\alpha\leq 1$.
We write $\vec G_{\alpha}(n, v)$ in place of $\vec G(n, \ceil{\alpha n}, v)$. Let $S_{\vec G_{\alpha}(n,v)}$ denote the sandpile group of the directed graph $\vec G_{\alpha}(n,  v)$.
Bhargava, DePascale and Koenig \cite{bhargava2023rank} conjectured the distribution of the Sylow $p$-subgroup of the sandpile group $S_{\vec G_{\alpha}(n,v)}$.

\begin{conjecture}\cite{bhargava2023rank}\label{Conj: directed bipartite graph}
Consider constants $0<v<1$ and $ \frac{1}{p}<\alpha \leq 1$ and a finite abelian $p$-group $G$.
Then
\[
\lim_{n \rightarrow \infty} 
\PPb{(S_{\vec G_{\alpha}(n,v)})_p \cong G} 
=P_{\infty,\frac{1}{p}}(G).
\]
\end{conjecture}

Note that the limit does not depend on the edge probability $v$. 
Our main result is to prove this conjecture.

\begin{theorem}\label{Thm: dist directed bipartite graph}
Consider constants $0<v<1$ and $ \frac{1}{p}<\alpha\leq 1$ and a finite abelian $p$-group $G$.
Then
\[
\lim_{n \rightarrow \infty} 
\PPb{(S_{\vec G_{\alpha}(n,v)})_p \cong G} 
=P_{\infty,\frac{1}{p}}(G).
\]  
\end{theorem}

The distribution of the $p$-ranks of $S_{\vec G_{\alpha}(n,v)}$ was determined by Bhargava, DePascale and Koenig \cite{bhargava2023rank}. It is shown in \cite{FulmanKaplanSinghalWarnaar_SylowSandpileBipartite} that this is consistent with Conjecture~\ref{Conj: directed bipartite graph}.

\begin{theorem}
Let $p$ be prime, $0<v<1$ and $ \frac{1}{p}<\alpha \leq 1$.
Then for every $r\geq 0$
\[
\lim_{n \rightarrow \infty} 
\PPb{\rank((S_{\vec G_{\alpha}(n,v)})_p) =r}
=\frac{1}{p^{r^2+r}} 
\frac{\qpochn{\frac{1}{p^{r+2}}}{\frac{1}{p}}{\infty}}{\qpochn{\frac{1}{p}}{\frac{1}{p}}{r}}.
\]
\end{theorem}

\subsection{Random Matrix model}

Let $\Z_p$ be the ring of $p$-adic integers. Fix $0<\epsilon<1$ for the rest of the paper.
A random variable $x$ taking values in $\Z_p$ is called \emph{$\epsilon$-balanced} if for each $a\in \Z/p\Z$,
\[
\PPP[x\equiv a\pmod{p}]<1-\epsilon.
\]

We consider a random matrix model for the Laplacian of $\vec G_{\alpha}(n,v)$ as a random matrix $\Lnot\in \MM_{n_1+n_2}(\Z_p)$ as follows:
\[
\Lnot
=\begin{bmatrix}
-\sum b_{1j} & 0&  \dots & 0 & a_{11}& a_{12} & \dots & a_{1n_2}\\
0 & -\sum b_{2j} &\dots  & 0 & a_{21}& a_{22} & \dots & a_{2n_2}\\
\vdots & &\ddots & \vdots & \vdots & & & \vdots\\
0 & 0 &\dots  & -\sum b_{n_1j} & a_{n_11}& a_{n_12} & \dots & a_{n_1n_2}\\
b_{11}& b_{21} & \dots & b_{n_11} &-\sum a_{i1} & 0&  \dots & 0 \\
b_{12}& b_{22} & \dots & b_{n_12} &0 & -\sum a_{i2} &\dots  & 0 \\
\vdots & & & \vdots &\vdots & &\ddots & \vdots \\
b_{1n_2}& b_{2n_2} & \dots & b_{n_1n_2} &0 & 0 &\dots  & -\sum a_{in_2}\\
\end{bmatrix},
\]
where all $a_{ij}$ and $b_{ij}$ are $\epsilon$-balanced and independent.
Now, $\Lnot\in\Hom(\Zpnot,\Zpnot)$. Let $Z_n$ consist of the vectors in $\Z_p^n$ whose entries sum to $0$, so $\im(\Lnot) \subseteq \Znot$. We denote $G(\Lnot)=\Znot/\im(\Lnot)$.
In particular, if $a_{ij}$ and $b_{ij}$ took values $1$ and $0$ with probability $v$ and $1-v$ respectively, then $\Lna$ would be the random matrix arising from the Laplacian of the random graph $\vec G_{\alpha}(n,v)$.
Moreover, in the (high probability) event that $\vec G_{\alpha}(n,v)$ is connected, the random finite abelian $p$-group $G(\Lna)$ would become isomorphic to the $p$-Sylow subgroup of $S_{\vec G_{\alpha}(n,v)}$.

One of the difficulties of working with $\Lnot$ is that the diagonal entries depend on other entries of the matrix. We therefore consider a related random matrix whose diagonal entries are independent.
This strategy of replacing the Laplacian model with an independent diagonal model was first used by Wood in the proof of \cite[Theorem 6.2]{wood2017distribution}.
Consider the random matrix $\Mnot\in \MM_{n_1+n_2}(\Z_p)$ defined as follows:
\[
\Mnot
=\begin{bmatrix}
c_1 & 0&  \dots & 0 & a_{11}& a_{12} & \dots & a_{1n_2}\\
0 & c_2 &\dots  & 0 & a_{21}& a_{22} & \dots & a_{2n_2}\\
\vdots & &\ddots & \vdots & \vdots & & & \vdots\\
0 & 0 &\dots  & c_{n_1} & a_{n_11}& a_{n_12} & \dots & a_{n_1n_2}\\
b_{11}& b_{21} & \dots & b_{n_11} &d_1 & 0&  \dots & 0 \\
b_{12}& b_{22} & \dots & b_{n_12} &0 & d_2 &\dots  & 0 \\
\vdots & & & \vdots &\vdots & &\ddots & \vdots \\
b_{1n_2}& b_{2n_2} & \dots & b_{n_1n_2} &0 & 0 &\dots  & d_{n_2}\\
\end{bmatrix},
\]
where all $a_{ij}$ and $b_{ij}$ are $\epsilon$-balanced and $c_j$ and $d_j$ are Haar uniform and mutually independent.
Note that $\Mnot\in\Hom(\Zp^{n_1+n_2},\Zp^{n_1+n_2})$, so we can consider the cokernel $\Coker(\Mnot)=\Zpnot/\im(\Mnot)$.

For $0<\alpha\leq 1$, we will denote $\Mna=M_{n,\aln}$ and $\Lna=L_{n,\aln}$.
We will see that $\Coker(\Mna)$ is easier to analyze than $G(\Lna)$, and that its analysis helps us study $G(\Lna)$.

We will prove the following:
\begin{theorem}\label{Thm: distribution ind diag}
Given $\frac{1}{p}<\alpha\leq 1$ and a finite abelian $p$-group $G$, we have
\[
\lim_{n\to\infty}
\PPb{\Coker(\Mna)\cong G}
=P_{\infty,1}(G).
\]    
\end{theorem}

\begin{theorem}\label{Thm: distribution Laplacian}
Given $\frac{1}{p}<\alpha\leq 1$ and a finite abelian $p$-group $G$, we have
\[
\lim_{n\to\infty}
\PPP[G(\Lna)\cong G]
=P_{\infty,\frac{1}{p}}(G).
\]    
\end{theorem}

Note that Theorem~\ref{Thm: distribution Laplacian} directly implies Theorem~\ref{Thm: dist directed bipartite graph}.

\subsection{Previous results}

Friedman and Washington \cite{friedman1989distribution} computed the limiting distribution of $\Coker(X)$, where $X$ is a Haar uniform element of $\mathbb{M}_n(\Zp)$.

\begin{theorem}\cite{friedman1989distribution}
Let $X_n \in \mathbb{M}_n(\Z_p)$ be Haar-random. Then for every finite abelian $p$-group $G$,
\[
\lim_{n\to\infty} \PPP[\Coker(X_n)\cong G] = P_{\infty,1}(G).
\]
\end{theorem}

This result gives one of the earliest random matrix realizations of the Cohen-Lenstra distribution. It shows that the weighting by $\abs{\Aut(G)}^{-1}$ arises naturally from the cokernels of Haar-random $p$-adic matrices.

Wood \cite{wood2019random} proved a stronger result for rectangular matrices with independent $\epsilon$-balanced entries.

\begin{theorem}\cite{wood2019random}\label{Thm: wood coker rectangle}
Fix a prime $p$ and integer $u\geq 0$. Let $X_n\in\mathbb{M}_{n,n+u}(\Zp)$ be a random matrix whose entries are independent and $\epsilon$-balanced.
Then for every finite abelian $p$-group $G$,
\[
\lim_{n\to\infty} 
\PPP[\Coker(X_n) \cong G] 
= P_{\infty,p^{-u}}(G).
\]
\end{theorem}

Nguyen and Wood \cite{nguyen2022random} proved a global version for random integral matrices over $\Z$.

\begin{theorem}\cite{nguyen2022random}
\label{Thm: Nguyen 1}
Fix an integer $u\ge 0$ and constants $0<\epsilon<1$ and $T>0$.
For each positive integer $n$, let $\xi_n\in\Z$ be a random integer that is $n^{-1+\epsilon}$-balanced for every prime $p$ and satisfies $\abs{\xi_n}\leq n^T$.
Let $X_n\in\mathbb{M}_{n,n+u}(\Z)$ be a random matrix whose entries are independent, identically distributed copies of the random integer $\xi_n$ defined above.
Then for every finite abelian group $B$,
\[
\lim_{n\to\infty} \PPP[\Coker(X_n)\cong B]
=
\frac{1}{\abs{B}^u \abs{\Aut(B)}}
\prod_{k=u+1}^{\infty}\zeta(k)^{-1}.
\]
\end{theorem}


M\'esz\'aros determined the distribution of the Sylow $p$-subgroups of sandpile groups of random $d$-regular directed graphs \cite{meszaros2020distribution}.
Suppose $d\ge 3$.  The random $d$-regular directed graph $D_n$ is obtained by choosing $d$ independent uniform random permutations $\pi_1,\dots,\pi_d$ of $\{1,\dots,n\}$ and, for each $v\in\{1,\dots,n\}$ and $1\le j\le d$, adding a directed edge from $v$ to $\pi_j(v)$.  Let $S_{D_n}$ denote the sandpile group of $D_n$.

\begin{theorem}\cite[Theorem~1.1]{meszaros2020distribution}\label{mes_thm_directed}
Consider a finite abelian $p$-group $G$. Then we have
\[
\lim_{n\to\infty}
\PPP\left( (S_{D_n})_{p} \cong G \right)
=
P_{\infty,1}(G).
\]
\end{theorem}

There are several further developments of this circle of ideas. Cheong and Kaplan \cite{cheong2022generalizations} studied cokernels of polynomial expressions $P(A)$, where $A$ is a Haar-random matrix over $\Zp$.
Cheong and Yu \cite{cheong2023distribution} proved the $\epsilon$-balanced analogue for cokernels of polynomial expressions $P(A)$.
Cheong and Huang \cite{cheong2025cokernel} studied the distribution of $P(A+pB)$, where $P$ is a polynomial expression, $A$ is a fixed matrix and $B$ is Haar-random.
Lee \cite{lee2023joint} studied the joint distribution of several cokernels of the form $\Coker(P_1(A)),\dots,\Coker(P_r(A))$, showing that one can study finer questions than the distribution of a single cokernel. 
Finally, Nguyen and Wood \cite{nguyen2025local} developed a local and global universality theory for more structured random matrix ensembles, including symmetric, skew-symmetric, and Laplacian-type models.

For symmetric matrices, the limiting distributions differ from the Cohen-Lenstra distributions $P_{\infty,u}$. This is because the cokernel of a symmetric matrix naturally carries an additional algebraic structure, namely a canonical symmetric bilinear pairing.
Clancy, Kaplan, Leake, Payne, and Wood \cite{clancy2015cohen} proved that Haar-random symmetric matrices over $\Zp$ give rise to the distribution
\[
\msympinf(G)
:=\frac{\abs{\{\text{symmetric, bilinear, perfect pairings }G\times G\to \C^{\times}\}}}{\abs{G}\abs{\Aut(G)}}
\prod_{k\geq 0}(1-p^{-1-2k}).
\]
Bhargava, Kane, Lenstra, Poonen, and Rains \cite{bhargava2015modeling} showed an analogous result for cokernels of Haar-random alternating matrices over $\Zp$.


Wood \cite{wood2017distribution} extended this result to random symmetric matrices, whose entries on and below the diagonal are independent and $\epsilon$-balanced.
Restricting her result to a single prime $p$, one obtains the following statement.

\begin{theorem}\cite{wood2017distribution}\label{Thm: Wood dist of sym matrix coker}
Let $X_n\in\mathbb{M}_n(\Zp)$ be a random symmetric matrix, whose entries on and below the diagonal are independent and $\epsilon$-balanced. Then for every finite abelian $p$-group $G$,
\[
\lim_{n\to\infty}
\PPP[\Coker(X_n)\cong G]
=
\msympinf(G).
\]
\end{theorem}

Hodges \cite{hodges2024distribution} proved the corresponding result for the distribution of the cokernel along with its pairing.
Shen \cite{shen2026quantative} later reproved the symmetric universality theorem by a different method and obtained effective error bounds.

For comparison, we also recall the corresponding results for undirected graphs. In the undirected case, the Laplacian is symmetric and the sandpile group carries a natural pairing.
Wood \cite{wood2017distribution} proved that the distribution of the Sylow $p$-subgroup of the sandpile group of a random Erd\H{o}s--R\'enyi undirected graph converges to the distribution $\msympinf$.
M\'esz\'aros \cite{meszaros2020distribution} proved an analogous result for random $d$-regular graphs.



In the companion paper \cite{Singhal_SandpileBipartite_Undirected}, we study random undirected bipartite graphs. For odd $p$, we determine the distribution of the $p$-Sylow subgroups of their sandpile groups. For $p=2$, where an interesting special distribution arises, we do not completely determine the distribution, but we prove several intermediate results.

We now briefly outline the paper. In Section~\ref{Sec: Strategy dir}, we explain the strategy of the proof. 
In Section~\ref{Sec: set ind diag dir} and Section~\ref{Sec: ind diag dir}, we study the independent diagonal model. 
In Section~\ref{Sec: set laplacian dir} and Section~\ref{Sec: laplacian dir}, we return to the Laplacian model. 
In Section~\ref{Sec: Raw moment infty}, we show that the expected number of surjections to a fixed $G$ can diverge to infinity. 
Finally, in Section~\ref{Sec: tech lemma dir}, we prove several technical lemmas used throughout the paper.

\section{Strategy}\label{Sec: Strategy dir}

Previous results about the distribution of sandpile groups of random graphs and cokernels of random matrices have been proven using the following strategy developed by Wood \cite{wood2017distribution}. Suppose $G_n$ is a sequence of random finite abelian $p$-groups and $\nu$ is the conjectured distribution as $n\to\infty$.
The first step is to show that the surjective moments match. That is, for every finite abelian $p$-group $H$, we have
\begin{equation}\label{Eqn: step 1}
\lim_{n\to\infty} \E[\abs{\Sur(G_n,H)}]
=\E[\abs{\Sur(G,H)}],
\end{equation}
where $G$ is a finite abelian $p$-group chosen according to $\nu$.
The next step is to use a result of Wood \cite{wood2017distribution}, which says that if \eqref{Eqn: step 1} is satisfied and the moments are not too big, then for every finite abelian $p$-group $H$
\[
\lim_{n\to\infty} \PPP[G_n\cong H]
=\nu(H).
\]

The surjective moments of $P_{\infty,u}$ were computed by Wood \cite{wood2019random}.

\begin{theorem}\cite[Lemma 3.2]{wood2019random}\label{Thm: moments of P infty u}
Suppose we are given a constant $0<u<p$ and $X$ is a random finite abelian $p$-group distributed according to $P_{\infty,u}$,
Then for every finite abelian $p$-group $G$, we have
\[
\E[\abs{\Sur(X, G)}]
=\abs{G}^{\log_p(u)}.
\]
\end{theorem}

Based on this, we should expect that for each finite abelian $p$-group $H$, we have
\[
\lim_{n\to\infty}
\E[\abs{\Sur(\Coker(\Mna),H)}]
=1,
\]
\[
\lim_{n\to\infty}
\E[\abs{\Sur(G(\Lna),H)}]
=\frac{1}{\abs{H}}.
\]
However, this is not always the case, as we will show in Section~\ref{Sec: Raw moment infty}. These limits sometimes blow up to infinity.
The issue here is that a small subset of matrices contribute greatly to the surjective moments while contributing minimally to the distribution. We can get around this problem by restricting ourselves to a subset of the matrices whose probability goes to $1$ as $n\to\infty$, and computing the moments on this subset.

\subsection{Independent diagonal model}

We denote
\begin{align}
\Gamma(\Mnot)&=\Big\{1\leq i\leq n_1\mid c_i\equiv 0\pmod{p}\Big\}
&\text{ and }&
&\gamma(\Mnot)=\abs{\Gamma(\Mnot)} .
\end{align}
Since the $c_i$ are Haar-uniform, we should expect $\gamma(\Mnot)$ to be concentrated around $\frac{1}{p} n_1$. Given a constant $\rho>0$, we consider the subset of matrices for which $\abs{\gamma(\Mnot) -\frac{1}{p}n_1}\leq\rho n_1$. We show in Corollary~\ref{Cor: prob of cond to 1} that the probability of this goes to $1$ as $n\to\infty$.

Given a function $f$ that takes $\Mnot$ as input, we denote the conditional expectation as
\[
\Erho[f(\Mnot)]
=\E[f(\Mnot) \mid |\gamma(\Mnot)-\tfrac{1}{p} n_1|\leq \rho n_1].
\]
We also denote the expectation restricted to the event $|\gamma(\Mnot)-\tfrac{1}{p} n_1|\leq \rho n_1$ as
\[
\Erhop[f(\Mnot)]
=\PPb{\abs{\gamma(\Mnot)-\tfrac{1}{p} n_1}\leq \rho n_1} \times \E_{\rho}[f(\Mnot)].
\]

We will show that once we restrict ourselves to this subset, we get the right surjective moments and the issue of moments blowing up to infinity is resolved.

\begin{proposition}\label{Prop: ind diag sur}
Given $\frac{1}{p}<\alpha\leq 1$, $0<\rho<\min(\frac{1}{p},\alpha-\frac{1}{p})$  and a finite abelian $p$-group $G$, we have
\[
\lim_{n\to\infty} \Erho[\abs{\Sur(\Coker(\Mna),G)}] =1.
\]
\end{proposition}

Note that the restriction $\rho<\alpha-\frac{1}{p}$ is to avoid the matrices for which $\gamma(\Mna)> n_2=\ceil{\alpha n}$. These are the matrices responsible for the large surjective moments.
If $\alpha=1$, then no matrix can have $\gamma(\Mna)> n_2=n$. In fact, if $\alpha=1$, then the issue of moments blowing up does not arise and one could work directly with the raw moments $\lim_{n\to\infty} \E[\big|\Sur(\Coker(M_n^{(1)}),G)\big|]$.

Once we have the moments, we can show that the desired distribution occurs using the following result of Wood.
\begin{theorem}\cite[Theorem 8.3]{wood2017distribution}\label{Thm: Wood universality}
If $X_n$ is a sequence of random finitely generated $\Zp$-modules and $Y$ is a random finitely generated $\Zp$-module such that for every finite abelian $p$-group $G$, we have
\[
\lim_{n\to\infty} \E[\abs{\Sur(X_n,G)}]
=\E[\abs{\Sur(Y,G)}]
\leq \abs{\Lambda^2G}.
\]
Then for every finite abelian $p$-group $G$, we have
\[
\lim_{n\to\infty}
\PPP[X_n\cong G]
=\PPP[Y\cong G].
\]    
\end{theorem}

We can now show that Proposition~\ref{Prop: ind diag sur} implies Theorem~\ref{Thm: distribution ind diag}.

\begin{customthm}{\ref{Thm: distribution ind diag}}
Given $\frac{1}{p}<\alpha\leq 1$ and a finite abelian $p$-group $G$, we have
\[
\lim_{n\to\infty}
\PPb{\Coker(\Mna)\cong G}
=P_{\infty,1}(G).
\]    
\end{customthm}
\begin{proof}[Proof of Theorem~\ref{Thm: distribution ind diag} assuming Proposition~\ref{Prop: ind diag sur}]
Fix $0<\rho<\min(\frac{1}{p},\alpha-\frac{1}{p})$.
Let $G_n$ be a random finite abelian $p$-group distributed according to
\[
\PPb{G_n\cong G}
= \PPb{ \Coker(\Mna)\cong G \mid \abs{\gamma(\Mna)-\tfrac{1}{p} n}\leq \rho n}.
\]
Therefore, Proposition~\ref{Prop: ind diag sur} says that
\[
\lim_{n\to\infty} \E[\abs{\Sur(G_n,G)}] =1.
\]
Therefore from Theorem~\ref{Thm: moments of P infty u} and Theorem~\ref{Thm: Wood universality} we see that for each finite abelian $p$-group $G$, we have
\[
\lim_{n\to\infty} \PPb{ \Coker(\Mna)\cong G \mid \abs{\gamma(\Mna)-\tfrac{1}{p} n} \leq \rho n}
=\lim_{n\to\infty} \PPP[G_n\cong G]
=P_{\infty,1}(G).
\]
Note that
\begin{align*}
&\PPP[\Coker(\Mna)\cong G]\\
&= \PPP[ \Coker(\Mna)\cong G \mid \abs{\gamma(\Mna)-\tfrac{1}{p} n}\leq \rho n]
\times \PPP[\abs{\gamma(\Mna)-\tfrac{1}{p} n}\leq \rho n]\\
&+\PPP[ \Coker(\Mna)\cong G \mid \abs{\gamma(\Mna)-\tfrac{1}{p} n}> \rho n]
\times \PPP[\abs{\gamma(\Mna)-\tfrac{1}{p} n}> \rho n].
\end{align*}
The result follows since Corollary~\ref{Cor: prob of cond to 1} says that $\lim_{n\to\infty} \PPP[\abs{\gamma(\Mna)-\tfrac{1}{p} n}\leq \rho n]=1$.
\end{proof}

Moreover, to prove Proposition~\ref{Prop: ind diag sur}, it is enough to compute the Hom-moments.

\begin{proposition}\label{Prop: ind diag hom}
Given $\frac{1}{p}<\alpha\leq 1$, $0<\rho<\min(\frac{1}{p},\alpha-\frac{1}{p})$  and a finite abelian $p$-group $G$, we have
\[
\lim_{n\to\infty}
\Erho[\abs{\Hom(\Coker(\Mna),G)}] =\abs{\{\text{Subgroups of }G\}}.
\]
\end{proposition}

We show that the computation of Hom-moments indeed leads to the surjective moments.

\begin{proof}[Proof of Proposition~\ref{Prop: ind diag sur}]
The result follows from Proposition~\ref{Prop: ind diag hom} by induction on the size of $G$.
\end{proof}

\subsection{Laplacian model}

Similarly to the independent diagonal model denote
\begin{align*}
\Gamma(\Lnot)&=\Big\{1\leq i\leq n_1\mid \sum_{j=1}^{n_2}b_{ij}\equiv 0\pmod{p}\Big\}
&\text{ and }&
&\gamma(\Lnot)=\abs{\Gamma(\Lnot)} .
\end{align*}
Even though $b_{ij}$ are not necessarily Haar uniform, as $n_2$ gets bigger $\sum_{j=1}^{n_2}b_{ij}$ would get close to being uniformly distributed $\mod p$. Therefore, we should still expect $\gamma(\Lnot)$ to be concentrated around $\frac{1}{p} n_1$. Given a constant $\rho>0$, we consider the subset of matrices for which $\abs{\gamma(\Lnot) -\frac{1}{p}n_1}\leq\rho n_1$. We show in Lemma~\ref{Lem: gamma Lna concentrate} that the probability of this goes to $1$ as $n\to\infty$.

Also as before, for a function $f$ that takes $\Lnot$ as input, we denote
\[
\Erho[f(\Lnot)]
=\E[f(\Lnot) \mid |\gamma(\Lnot)-\tfrac{1}{p} n_1|\leq \rho n_1],
\]
and
\[
\Erhop[f(\Lnot)]
=\PPP[|\gamma(\Lnot)-\tfrac{1}{p} n_1|\leq \rho n_1] \times \E_{\rho}[f(\Lnot)].
\]

Once we restrict ourselves to this subset, then the issue of moments blowing up to infinity gets resolved and we get the right moments.

\begin{proposition}\label{Prop: Sur moment Laplacian}
Given $\frac{1}{p}<\alpha\leq 1$, $0<\rho<\min(\frac{1}{p},\alpha-\frac{1}{p})$  and a finite abelian $p$-group $G$, we have
\[
\lim_{n\to\infty} 
\Erho[\abs{\Sur(G(\Lna),G)}] =\frac{1}{|G|}.
\]
\end{proposition}

As in the independent diagonal model, the restriction $\rho<\alpha-\frac{1}{p}$ is to avoid the matrices for which $\gamma(\Lna)> n_2=\ceil{\alpha n}$. These are the matrices responsible for the large surjective moments.
These matrices correspond to those bipartite directed graphs for which the number of vertices in $V_1$ whose in-degree is divisible by $p$ is larger than $\abs{V_2}=\ceil{\alpha n}$.
If $\alpha=1$, then no matrix can have $\gamma(\Lna)> n_2=n$. In fact, if $\alpha=1$, then the issue of moments blowing up does not arise and one could work directly with the raw moments $\lim_{n\to\infty} \E[\big|\Sur(G(L_n^{(1)}),G)\big|]$.

We show that Proposition~\ref{Prop: Sur moment Laplacian} implies Theorem~\ref{Thm: distribution Laplacian}.
\begin{customthm}{\ref{Thm: distribution Laplacian}}
Given $\frac{1}{p}<\alpha\leq 1$ and a finite abelian $p$-group $G$, we have
\[
\lim_{n\to\infty}
\PPP[G(\Lna)\cong G]
=P_{\infty,\frac{1}{p}}(G).
\]
\end{customthm}
\begin{proof}[Proof of Theorem~\ref{Thm: distribution Laplacian} assuming Proposition~\ref{Prop: Sur moment Laplacian}]
Fix $0<\rho<\min(\frac{1}{p},\alpha-\frac{1}{p})$.
Proposition~\ref{Prop: Sur moment Laplacian}, Theorem~\ref{Thm: moments of P infty u} and Theorem~\ref{Thm: Wood universality} imply that for each finite abelian $p$-group $G$, we have
\[
\lim_{n\to\infty} \PPP[ G(\Lna)\cong G \mid \abs{\gamma(\Lna)-\tfrac{1}{p} n} \leq \rho n]
=P_{\infty,\frac{1}{p}}(G).
\]
Note that
\begin{align*}
&\PPP[G(\Lna)\cong G]\\
&= \PPP[ G(\Lna)\cong G \mid \abs{\gamma(\Lna)-\tfrac{1}{p} n}\leq \rho n]
\times \PPP[\abs{\gamma(\Lna)-\tfrac{1}{p} n}\leq \rho n]\\
&+\PPP[ G(\Lna)\cong G \mid \abs{\gamma(\Lna)-\tfrac{1}{p} n}> \rho n]
\times \PPP[\abs{\gamma(\Lna)-\tfrac{1}{p} n}> \rho n].
\end{align*}
The result follows since Lemma~\ref{Lem: gamma Lna concentrate} says that $\lim_{n\to\infty} \PPP[\abs{\gamma(\Lna)-\tfrac{1}{p} n}\leq \rho n]=1$.
\end{proof}

Therefore, our goal now is to prove Proposition~\ref{Prop: ind diag hom} and Proposition~\ref{Prop: Sur moment Laplacian}.

\section{Setup for Independent diagonal model}\label{Sec: set ind diag dir}

Fix a finite abelian $p$-group $G$ for this and the next section.
Similarly to \cite{wood2017distribution}, we can decompose the expected number of homomorphisms to $G$ as a sum of probabilities.
Since $\Mnot\in \Hom(\Zpnot,\Zpnot)$, we have
\begin{align*}
&\Erhop[\abs{\Hom(\Coker(\Mnot),G)}]\\
&=\sum_{F\in\Hom(\Zpnot,G)} 
\PPb{F\circ \Mnot=0\in \Hom(\Zpnot,G) \text{ and } \abs{\gamma(\Mnot) -\tfrac{1}{p} n_1}\leq \rho n_1}.
\end{align*}

For each $F\in\Hom(\Zpnot,G)$,
\begin{equation}\label{Eqn: Probabilty as sum over S}
\begin{split}
&\PPP[F\circ \Mnot=0 \text{ and } \abs{\gamma(\Mnot) -\tfrac{1}{p} n_1}\leq \rho n_1]\\
&=
\sum_{\substack{S\subseteq [n_1]:\\ \abs{\abs{S} -\frac{1}{p}n_1}\leq \rho n_1}}
\PPP[F\circ \Mnot=0 \text{ and } \Gamma(\Mnot)=S].   
\end{split}
\end{equation}

Denote $A_j=\begin{bmatrix}
a_{1j}\\ \vdots \\ a_{n_1j}
\end{bmatrix}\in \Zp^{n_1}$,
$B_i=\begin{bmatrix}
b_{i1}\\ \vdots \\ b_{in_2}
\end{bmatrix}\in \Zp^{n_2}$.
Denote the standard basis of $\Zp^{n_1}$ as $u_1,\dots,u_{n_1}$ and that of $\Zp^{n_2}$ as $v_1,\dots, v_{n_2}$.
Also note that $F$ can be decomposed into $F_1\in\Hom(\Zp^{n_1},G)$ and $F_2\in\Hom(\Zp^{n_2},G)$.
Since the columns of $\Mnot$ are independent of each other, we see that
\begin{equation}\label{Eqn: Probabilty as product}
\begin{split}
&\PPP[(F_1,F_2)\circ \Mnot=0 \text{ and }\Gamma(\Mnot)=S]\\
&=\prod_{i\in S} 
\PPP[ F_2(B_i)=-c_iF_1(u_i) \text{ and } c_i\equiv 0\pmod{p}]\\
&\quad\quad\times
\prod_{i\in [n_1]\setminus S} 
\PPP[F_2(B_i)=-c_iF_1(u_i)
\text{ and } c_i\not\equiv 0\pmod{p}]\\
&\quad\quad\times
\prod_{j\in [n_2]} 
\PPP[F_1(A_j)=-d_jF_2(v_j)].
\end{split}
\end{equation}

We therefore want to estimate the probabilities appearing in \eqref{Eqn: Probabilty as product}. The vectors $A_j\in \Zp^{n_1}$ and $B_i\in \Zp^{n_2}$ have independent, $\epsilon$-balanced entries. We will first compute such probabilities, in the case when the entries are Haar-uniform.

\begin{lemma}\label{Lem: Haar uniform}
Suppose we are given $F\in \Hom(\Zp^n,G)$ with $\im(F)=H$. Let $X\in \Zp^n$ be a random vector whose entries are independent and Haar-uniform. 
Then for each $h\in H$, we have
\[
\PPP\big[F(X)= h\big]
=\frac{1}{\abs{H}}.
\]
\end{lemma}
\begin{proof}
Note that for each $g\in H$, $F^{-1}(g)$ is a translation of $\ker(F)$. Since the Haar measure is translation invariant, each $F^{-1}(g)$ must have the same Haar measure. Finally, since the measure of $\Zp^n$ is $1$, the measure of each $F^{-1}(g)$ is $\frac{1}{\abs{H}}$.
\end{proof}

In the case where the entries of $X$ are $\epsilon$-balanced but not Haar-uniform, then we cannot compute the probability with such a simple method.
In the next subsection, we will describe how to partition $\Hom(\Zpnot,G)$, how to estimate the probabilities for $F$ in certain parts and how to bound them for $F$ in other parts.

\subsection{Codes and depth}

In this subsection, we recall some material due to Wood \cite{wood2017distribution,wood2019random}.
Let $V=\Z_p^n$ with the standard basis $\{v_1,\dots,v_n\}$. Denote $[n]=\{1,\dots,n\}$. Given a subset $\sigma\subseteq [n]$, we have a distinguished submodule $\Vsigma\subseteq V$, which is generated by $v_i$ satisfying $i\notin \sigma$ and another distinguished submodule $V_{\sigma}\subseteq V$, which is generated by $v_i$ satisfying $i\in \sigma$. The notions of code and $\delta$-depth were defined by Wood \cite{wood2019random}. Our definitions are slightly different from hers.

\begin{definition}
Consider $F\in \Hom(V,G)$. It is called a code of distance $d$ with image $H$ if for every subset $\sigma\subseteq [n]$ with $\abs{\sigma}<d$, we have $F(\Vsigma)=F(V)=H$.
\end{definition}
This means not only is $F$ surjective to $H$, but it remains surjective to $H$ as long as you do not throw away too many basis vectors.
If $\im(F)=H$ and $F$ is not a code, it means we can throw away a few basis vectors and shrink the image. The depth of $F$ measures the extent to which we can shrink the image without throwing away too many basis vectors.

\begin{definition}
Given $\delta>0$ and $F\in\Hom(V,G)$, the $\delta$-depth of $F$ is defined to be the largest positive integer $D$ for which there is a subset $\sigma\subseteq [n]$ with $\abs{\sigma}<\log_p(D)\delta n$ such that $[F(V):F(\Vsigma)]=D$, or $1$ if there is no such $\sigma$.

If such a $\sigma$ exists, then we say that $F$ is of $\delta$-type $(F(V),F(\Vsigma))$. If there is no such $\sigma$, then we set the $\delta$-type to be $(F(V), F(V))$.
\end{definition}

It is noted in \cite[Remark 2.5]{wood2019random} that if $F$ has $\delta$-depth $1$, then it is a code of distance $\delta n$.

Given subgroups $T\subseteq H\subseteq G$, we denote
\[
\SSS{n,\delta}{H}{T}
:=\{F\in \Hom(\Zp^n,G): \delta-\text{type }(H,T)\}.
\]
This gives a cover
\[
\Hom(\Zp^{n},G)
=\bigcup_{H\subseteq G}
\bigcup_{T\subseteq H}
\SSS{n,\delta}{H}{T}.
\]
In general, this union need not be disjoint, since a given $F$ may have $\delta$-type $(H,T)$ for more than one choice of $T$. However, the pieces $\SSS{n,\delta}{H}{H}$ consist of codes with image $H$, and these pieces are pairwise disjoint.

First consider the case $T=H$, so all $F\in \SSS{n,\delta}{H}{H}$ are codes of distance $\delta n$ with image $H$. For such $F$, Wood shows that the probability is close to $\frac{1}{\abs{H}}$ even when $X$ is $\epsilon$-balanced.

\begin{lemma}\cite[Lemma 2.1]{wood2019random}\label{Lem: Wood estimate code}
Suppose $F\in \Hom(\Zp^n,G)$ is a code of distance $d$ with image $H$. Let $X\in\Zp^n$ be a random vector whose entries are independent and $\epsilon$-balanced.
Then for every $f\in H$, we have
\[
\left| \PPP[F(X)=f]-\frac{1}{\abs{H}}\right|
\leq \exp\Big(-\frac{\epsilon d}{|G|^2}\Big).
\]
\end{lemma}

Next consider the case $T\subsetneq H$. For $F\in \SSS{n,\delta}{H}{T}$, the probability can be bigger than $\frac{1}{\abs{H}}$, but  Wood shows that it can still be bounded as follows.

\begin{lemma}\cite[Lemma 2.7]{wood2019random}\label{Lem: Wood bound depth}
Suppose we are given subgroups $T\subsetneq H\subseteq G$,
Let $X\in\Zp^n$ be a random vector whose entries are independent and $\epsilon$-balanced.
Then for every $F\in \SSS{n,\delta}{H}{T}$ and $f\in G$, we have
\[
\PPP[F(X)=f]
\leq (1-\epsilon) 
 \left(\frac{1}{\abs{T}}+\exp\Big(-\frac{\epsilon\delta n}{|G|^2}\Big) \right).
\]
\end{lemma}
\begin{proof}
Actually the result in \cite[Lemma 2.7]{wood2019random} is for $f=0$, but the same proof goes through for any $f\in G$.
\end{proof}

\begin{corollary}\label{Cor: combined bound}
Suppose we are given $\delta>0$ and subgroups $T\subseteq H\subseteq G$.
Denote
\[
e=\begin{cases}
1&\text{if } T=H\\
1-\epsilon &\text{if }T\subsetneq H.
\end{cases}
\]
Let $X\in\Zp^n$ be a random vector whose entries are independent and $\epsilon$-balanced.
Then for every $F\in\SSS{n,\delta}{H}{T}$ and for every $f\in G$, we have
\[
\PPP[F(X)=f]
\leq e 
 \left(\frac{1}{\abs{T}}+\exp\Big(-\frac{\epsilon\delta n}{|G|^2}\Big) \right).
\]
\end{corollary}
\begin{proof}
First, consider the case $T=H$, so $F$ is a code of distance $\delta n$ with image $T$. Now, if $f\notin T$, then the probability is zero and the inequality holds trivially, whereas if $f\in T$, then the result follows from Lemma~\ref{Lem: Wood estimate code}.

Next, consider the case $T\subsetneq H$. Then the result follows from Lemma~\ref{Lem: Wood bound depth}.
\end{proof}

We will also need a bound for the probability that $F(X)$ belongs to another subgroup.

\begin{lemma}\label{Lem: bound subgroup}
Suppose we are given $\delta>0$, subgroups $T\subseteq H\subseteq G$ and another subgroup $T'\subseteq G$. 
Let $X\in\Zp^n$ be a random vector whose entries are independent and $\epsilon$-balanced.
Then for every $F\in \SSS{n,\delta}{H}{T}$, we have
\[
\PPP[F(X)\in T']
\leq 
 \left(\frac{\abs{T\cap T'}}{\abs{T}}+\exp\Big(-\frac{\epsilon\delta n}{|G|^2}\Big) \right).
\]
\end{lemma}
\begin{proof}
Denote $D=[H:T]$.
Consider the map $\pi:H\twoheadrightarrow H/(H\cap T')$. So $\pi\circ F\in \Hom(\Zp^n,  H/(H\cap T'))$. Moreover, $F(X)\in T'$ if and only if $(\pi\circ F)(X)=0$.
There is some $\sigma\subseteq [n]$ with $\abs{\sigma}\leq \delta n \log_p(D)$ such that $F(V_{\setminus\sigma})=T$.

Also,
\[F\in\Hom(V_{\setminus\sigma}\oplus V_{\sigma},G)= \Hom(V_{\setminus\sigma},G) \oplus \Hom(V_{\sigma},G),\]
so we can decompose $F=(F_{\setminus\sigma}, F_{\sigma})$ and $X=(X_{\setminus\sigma}, X_{\sigma})$. Note that
\[(\pi\circ F)(X) =(\pi\circ F_{\setminus\sigma})(X_{\setminus\sigma})+(\pi\circ F_{\sigma})(X_{\sigma}).\]

We will show that $\pi\circ F_{\setminus\sigma}$ is a code of distance $\delta n$ with image $\pi(T)$. Assume for the sake of contradiction that this is not true. This means there is some proper subgroup $L\subsetneq \pi(T)$ and $\tau\subseteq [n]\setminus \sigma$ with $\abs{\tau}<\delta n$ such that $(\pi\circ F)(V_{\setminus(\sigma\cup\tau)})\subseteq L$.
Hence $F(V_{\setminus(\sigma\cup\tau)})\subseteq \pi^{-1}(L)$. We also know that $F(V_{\setminus(\sigma\cup\tau)})\subseteq F(V_{\setminus\sigma})=T $, and hence $F(V_{\setminus(\sigma\cup\tau)})\subseteq \pi^{-1}(L)\cap T$.
Since $L$ is a proper subgroup of $\pi(T)$, it follows that $\pi^{-1}(L)\cap T$ is a proper subgroup of $T$. This means that $D'=[H:\pi^{-1}(L)\cap T]\geq pD$. Moreover, $\abs{\sigma\cup \tau} < \delta n\log_p(D)+\delta n \leq \delta n\log_p(D')$. However, this implies that the $\delta$-depth of $F$ is at least $D'$, which is a contradiction. We conclude that $\pi\circ F_{\setminus\sigma}$ is a code of distance $\delta n$ with image $\pi(T)$.

Therefore, by Lemma~\ref{Lem: Wood estimate code}, we see that
\begin{align*}
\PPP[F(X)\in T']
&=\PPP[(\pi\circ F)(X)=0]
=\PPP[(\pi\circ F_{\setminus\sigma})(X_{\setminus\sigma})=-(\pi\circ F_{\sigma})(X_{\sigma})]\\
&=\PPb{(\pi\circ F_{\setminus\sigma})(X_{\setminus\sigma})=-(\pi\circ F_{\sigma})(X_{\sigma}) \mid (\pi\circ F_{\sigma})(X_{\sigma})\in \pi(T)} 
\PPP[(\pi\circ F_{\sigma})(X_{\sigma})\in \pi(T)]\\
&\leq \left(\frac{1}{\abs{\pi(T)}} + \exp\Big(-\frac{\epsilon\delta n}{|G|^2}\Big)\right) \times 1.
\end{align*}
The result follows since $\abs{\pi(T)}=\frac{\abs{T}}{\abs{T\cap T'}}$.
\end{proof}

We will also need bounds on the sizes of $\SSS{n,\delta}{H}{T}$. We start with the case $T\subsetneq H$.

\begin{lemma}\cite[Lemma 5.2]{wood2017distribution}
Given subgroups $T\subsetneq H\subseteq G$, denote $D=[H:T]>1$. Then we have
\[
\abs{\SSS{n,\delta}{H}{T}}
\leq\binom{n}{\ceil{\delta\log_p(D) n}-1} \abs{T}^n D^{\delta\log_p(D)n}.
\]
\end{lemma}

Denote the binary entropy function as $H(\alpha)=-\alpha\log_2(\alpha)-(1-\alpha)\log_2(1-\alpha)$. It is known that $\binom{a}{b}\leq 2^{H(\frac{b}{a})a}$ and that $H(\alpha)\leq 2\sqrt{\alpha}$. It follows that $\binom{a}{b}\leq 2^{2\sqrt{ab}}$.
The following corollary follows at once.
\begin{corollary}
Given $D>1$ and subgroups $T\subsetneq H\subseteq G$ with $[H:T]=D$, we have
\[
\abs{\SSS{n,\delta}{H}{T}}
\leq 2^{2n\sqrt{\delta\log_p(\abs{G})}}
\abs{T}^n D^{\delta\log_p(D)n}.
\]
\end{corollary}

Given a group $G$, let $K_G$ be the number of subgroups of $G$.
We estimate the size of $\SSS{n,\delta}{H}{H}$.

\begin{lemma}\label{Lem: count codes}
Consider $0<\delta<\frac{1}{6\log_p(\abs{G})}$.
There is a constant $c_{G,\delta}>0$ (depending on $G$ and $\delta$), such that for each subgroup $H\subseteq G$
\[
\abs{H}^n \Big(1- K_G \exp(-c_{G,\delta}n) - K_G 2^{-n}\Big)
\leq \abs{\SSS{n,\delta}{H}{H}}
\leq \abs{H}^n.
\]
\end{lemma}
\begin{proof}
The upper bound follows from the fact that
\[
\SSS{n,\delta}{H}{H} \subseteq \Hom(\Zp^n, H).
\]
Moreover,
\[
\Hom(\Zp^n, H) \setminus \SSS{n,\delta}{H}{H}
=\bigcup_{T\subsetneq H} \SSS{n,\delta}{H}{T} \cup \Hom(\Zp^n,T).
\]
The number of $T$ is at most $K_G$, and 
\[
\abs{\Hom(\Zp^n,T)}
= \abs{T}^n \leq \frac{\abs{H}^n}{2^n}.
\]
Denote $[H:T]=D>1$. Then we know that
\[
\abs{\SSS{n,\delta}{H}{T}}
\leq 2^{2n\sqrt{\delta\log_p(\abs{G})}}
\abs{H}^n D^{-n+\delta\log_p(D)n}
\leq 
\abs{H}^n \Big(2^{2\sqrt{\delta\log_p(\abs{G})}-1+\delta\log_p(\abs{G})}\Big)^n.
\]
The bound on $\delta$ also ensures that $2\sqrt{\delta\log_p(\abs{G})}-1+\delta\log_p(\abs{G})<0$. The result follows.
\end{proof}

\section{Independent diagonal model}\label{Sec: ind diag dir}

We have fixed a finite abelian $p$-group $G$.
Recall that given $\delta>0$, we can partition
\[
\Hom(\Zp^{n_1},G)
=\bigcup_{H_1\subseteq G}
\bigcup_{T_1\subseteq H_1}
\SSS{n_1,\delta}{H_1}{T_1},
\]
\[
\Hom(\Zp^{n_2},G)
=\bigcup_{H_2\subseteq G}
\bigcup_{T_2\subseteq H_2}
\SSS{n_2,\delta}{H_2}{T_2}.
\]
We will show the following.
\begin{proposition}\label{Prop: ind diagonal pieces}
Given $\frac{1}{p}<\alpha\leq 1$ and $0<\rho<\min(\frac{1}{p},\alpha-\frac{1}{p})$ and a finite abelian $p$-group $G$, there is $\delta>0$ such that for any subgroups $T_1\subseteq H_1\subseteq G$, $T_2\subseteq H_2\subseteq G$, we have
\begin{align*}
&\lim_{n\to\infty}
\sum_{F_1\in \SSS{n,\delta}{H_1}{T_1}}
\sum_{F_2\in \SSS{\ceil{\alpha n},\delta}{H_2}{T_2}}
\PPP[(F_1,F_2)\circ \Mna=0 \text{ and } \abs{\gamma(\Mna)-\tfrac{1}{p}n} \leq \rho n]\\
&\quad\quad\quad=\begin{cases}
1 &\text{if } T_1=H_1=T_2=H_2\\
0 &\text{otherwise}.
\end{cases}
\end{align*}
\end{proposition}
This will imply Proposition~\ref{Prop: ind diag hom}.

\begin{customprop}{\ref{Prop: ind diag hom}}
Given $\frac{1}{p}<\alpha\leq 1$, $0<\rho<\min(\frac{1}{p},\alpha-\frac{1}{p})$  and a finite abelian $p$-group $G$, we have
\[
\lim_{n\to\infty}
\Erho[\abs{\Hom(\Coker(\Mna),G)}] =\abs{\text{Subgroups of }G}.
\]    
\end{customprop}
\begin{proof}[Proof of Proposition~\ref{Prop: ind diag hom} assuming Proposition~\ref{Prop: ind diagonal pieces}]
Proposition~\ref{Prop: ind diagonal pieces} implies that
\[
\lim_{n\to\infty}
\Erhop[\abs{\Hom(\Coker(\Mna),G)}] =\abs{\text{Subgroups of }G}.
\]
The result follows from Corollary~\ref{Cor: prob of cond to 1}.
\end{proof}

On each piece, our goal is to estimate
\[
\PPP[(F_1,F_2)\circ \Mnot=0 
\text{ and } 
\abs{\gamma(\Mnot) -\tfrac{1}{p} n_1}\leq \rho n_1].
\]
This probability can be decomposed as shown in \eqref{Eqn: Probabilty as sum over S} and \eqref{Eqn: Probabilty as product}.

\subsection{Main term}

In this subsection, the only requirement we put on $\delta$ is that $0<\delta<\frac{1}{3\log_p(\abs{G})}$ and we focus on the case when $T_1=H_1=T_2=H_2$. Denote this subgroup as $H$.
This means that $F_1$ is a code of distance $\delta n_1$ with image $H$ and $F_2$ is a code of distance $\delta n_2$ with image $H$.

We start by estimating the probability that $(F_1,F_2)\circ \Mnot=0$ and $\Gamma(\Mnot)=S$ for a fixed subset $S\subseteq [n_1]$.

\begin{lemma}\label{Lem: both codes, fixed S}
Suppose $\frac{1}{p}<\alpha\leq 1$,
$\alpha n_1\leq n_2\leq n_1$ and $S\subseteq[n_1]$.
If $H$ is a subgroup of $G$, $F_1\in\Hom(\Zp^{n_1},G)$ is a code of distance $\delta n_1$ with image $H$ and $F_2\in\Hom(\Zp^{n_2},G)$ is a code of distance $\delta n_2$ with image $H$. Then, for sufficiently large values of $n_1$ (depending on $G,\epsilon,\delta,\alpha$), we have
\begin{align*}
&\left|
\PPP[(F_1,F_2)\circ \Mnot=0 \text{ and }\Gamma(\Mnot)=S] 
- \frac{1}{p^{|S|}} \Big(1-\frac{1}{p}\Big)^{n_1-|S|} \frac{1}{|H|^{n_1+n_2}}
\right|\\
&\leq \frac{1}{p^{|S|}} \Big(1-\frac{1}{p}\Big)^{n_1-|S|} 4n_1\frac{1}{\abs{H}^{n_1+n_2-1}} \exp\Big(-\frac{\epsilon\delta \alpha n_1}{|G|^2}\Big).
\end{align*}
\end{lemma}
\begin{proof}
We estimate each term in the product given by \eqref{Eqn: Probabilty as product}. First, for $i\in S$, 
we know that
\begin{align*}
&\PPP[ F_2(B_i)=-c_iF_1(u_i)\text{ and } c_i\equiv 0\pmod{p}]\\
&=\PPP[c_i\equiv 0\pmod{ p}]
\PPP[F_2(B_i) =- c_iF_1(u_i) \mid c_i\equiv 0\pmod{ p}]\\
&=\frac{1}{p} \PPP[F_2(B_i) =- c_iF_1(u_i) \mid c_i\equiv 0\pmod{ p}].
\end{align*}
Moreover, Lemma~\ref{Lem: Wood estimate code} tells us that
\[
\left|
\PPP[ F_2(B_i)=- c_iF_1(u_i)\mid  c_i\equiv 0\pmod{p}]
- \frac{1}{\abs{H}}
\right|
\leq \exp\Big(-\frac{\epsilon\delta n_2}{|G|^2}\Big)
\leq \exp\Big(-\frac{\epsilon\delta \alpha n_1}{|G|^2}\Big).
\]
For $i\in [n_1]\setminus S$, we similarly see that
\begin{align*}
&\PPP[ F_2(B_i)=-c_iF_1(u_i)\text{ and } c_i\not\equiv 0\pmod{p}]\\
&=\PPP[c_i\not\equiv 0\pmod{ p}]
\PPP[F_2(B_i) =- c_iF_1(u_i) \mid c_i\not\equiv 0\pmod{ p}]\\
&=(1-\tfrac{1}{p}) \PPP[F_2(B_i) =- c_iF_1(u_i) \mid c_i\not\equiv 0\pmod{ p}].
\end{align*}
Moreover, Lemma~\ref{Lem: Wood estimate code} tells us that
\[
\left|
\PPP[ F_2(B_i)=- c_iF_1(u_i)\mid  c_i\not\equiv 0\pmod{p}]
- \frac{1}{\abs{H}}
\right|
\leq \exp\Big(-\frac{\epsilon\delta n_2}{|G|^2}\Big)
\leq \exp\Big(-\frac{\epsilon\delta \alpha n_1}{|G|^2}\Big).
\]
Finally for $j\in[n_2]$, Lemma~\ref{Lem: Wood estimate code} tells us that
\[
\left|
\PPP[F_1(A_j)=-d_jF_2(v_j)]
-\frac{1}{\abs{H}}
\right|
\leq \exp\Big(-\frac{\epsilon\delta n_1}{|G|^2}\Big)
\leq \exp\Big(-\frac{\epsilon\delta \alpha n_1}{|G|^2}\Big).
\]

Now, in order to apply Corollary~\ref{Cor: estimate product}, we need
\[
\frac{\exp\big(-\frac{\epsilon\delta \alpha n_1}{|G|^2}\big)}{2^{\frac{1}{2n_1}}-1}
\leq \frac{1}{\abs{H}}.
\]
This inequality holds for sufficiently large $n_1$, since the limit of the left hand side is $0$. Therefore, Corollary~\ref{Cor: estimate product} tells us that
\begin{align*}
&\left|
\frac{1}{\frac{1}{p^{\abs{S}}}(1-\frac{1}{p})^{n_1-\abs{S}}} 
\PPP[(F_1,F_2)\circ \Mnot=0 \text{ and }\Gamma(\Mnot)=S] 
-  \frac{1}{|H|^{n_1+n_2}}
\right|\\
&\leq 2(n_1+n_2)\frac{1}{\abs{H}^{n_1+n_2-1}} \exp\Big(-\frac{\epsilon\delta \alpha n_1}{|G|^2}\Big).
\end{align*}
The result follows.
\end{proof}

Next, we sum over the different subsets $S\subseteq[n_1]$ that satisfy $\abs{\abs{S}-\frac{1}{p}n_1}\leq \rho n_1$.

\begin{corollary}\label{Cor: both codes probability}
Suppose $\frac{1}{p}<\alpha\leq 1$, $0<\rho<\min(\frac{1}{p},\alpha-\frac{1}{p})$, $\alpha n_1\leq n_2\leq n_1$ and $H$ is a subgroup of $G$.
Suppose that $F_1\in\Hom(\Zp^{n_1},G)$ is a code of distance $\delta n_1$ with image $H$ and $F_2\in\Hom(\Zp^{n_2},G)$ is a code of distance $\delta n_2$ with image $H$.
Then, for sufficiently large values of $n_1$ (depending on $G,\epsilon,\delta,\alpha$), we have
\begin{align*}
&\left|
\PPP[(F_1,F_2)\circ \Mnot=0 \text{ and }
\abs{\gamma(\Mnot)- \tfrac{1}{p} n_1}\leq \rho n_1] 
- \frac{1}{\abs{H}^{n_1+n_2}}
\right|\\
&\leq 4n_1\frac{1}{\abs{H}^{n_1+n_2-1}} \exp\Big(-\frac{\epsilon\delta \alpha n_1}{|G|^2}\Big)
+ \frac{2}{\abs{H}^{n_1+n_2}} \exp\big(-2\rho^2 n_1\big).
\end{align*}
\end{corollary}
\begin{proof}
We sum the estimate of Lemma~\ref{Lem: both codes, fixed S} over the different $S\subseteq[n_1]$ with $\abs{\abs{S} - \frac{1}{p}n_1} \leq \rho n_1$,
\begin{align*}
&\left|
\PPP[(F_1,F_2)\circ \Mnot=0 \text{ and }
\abs{\gamma(\Mnot) -\tfrac{1}{p}n_1}\leq \rho n_1] 
- \frac{1}{\abs{H}^{n_1+n_2}}\!\!\!\!\!\!
\sum_{k:\abs{k-\frac{1}{p}n_1}\leq \rho n_1} \!\!\!\!\!\!\!\!
\tbinom{n_1}{k} \tfrac{1}{p^{k}} \big(1-\tfrac{1}{p}\big)^{n_1-k}
\right|\\
&\leq 4n_1\frac{1}{\abs{H}^{n_1+n_2-1}} \exp\Big(-\frac{\epsilon\delta \alpha n_1}{|G|^2}\Big).
\end{align*}
Lemma~\ref{Lem: bound tail} implies that
\[
\left|
\frac{1}{\abs{H}^{n_1+n_2}}\!\!\!\!\!\!
\sum_{k:\abs{k-\frac{1}{p}n_1}\leq \rho n_1} \!\!\!\!\!\!\!\!
\binom{n_1}{k} \frac{1}{p^{k}} \Big(1-\frac{1}{p}\Big)^{n_1-k}
-\frac{1}{\abs{H}^{n_1+n_2}}
\right|
\leq \frac{2}{\abs{H}^{n_1+n_2}} \exp\big(-2\rho^2 n_1\big).\qedhere
\]
\end{proof}

We now sum over the different $F_1\in \SSS{n,\delta}{H}{H}$, $F_2\in \SSS{n,\delta}{H}{H}$. Here we will use our estimate for the size of $\SSS{n,\delta}{H}{H}$.

\begin{proposition}\label{Prop: ind diag main term}
Suppose $\frac{1}{p}<\alpha\leq 1$, $0<\rho<\min(\frac{1}{p},\alpha-\frac{1}{p})$, $H$ is a subgroup of $G$ and $\delta$ is a constant satisfying $0<\delta<\frac{1}{3\log_p(\abs{G})}$.
Then, we have
\[
\lim_{n\to\infty}
\sum_{F_1\in \SSS{n,\delta}{H}{H}}
\sum_{F_2\in \SSS{\ceil{\alpha n},\delta}{H}{H}}
\PPP[(F_1,F_2)\circ \Mna=0 \text{ and }
\abs{\gamma(\Mna) -\tfrac{1}{p} n}\leq \rho n] 
= 1.
\]
\end{proposition}
\begin{proof}
Let $n_1=n$ and $n_2=\ceil{\alpha n}$.
Corollary~\ref{Cor: both codes probability} implies that
\begin{align*}
&\Bigg|
\left(
\sum_{F_1\in \SSS{n_1,\delta}{H}{H}}
\sum_{F_2\in \SSS{n_2,\delta}{H}{H}}
\PPP[(F_1,F_2)\circ \Mnot=0 \text{ and }
\abs{\gamma(\Mnot) -\tfrac{1}{p} n_1}\leq \rho n_1] 
\right)\\
&\quad\quad\quad\quad- \frac{\abs{\SSS{n_1,\delta}{H}{H}}\abs{\SSS{n_2,\delta}{H}{H}}}{\abs{H}^{n_1+n_2}}
\Bigg|\\
&\leq 4n_1\abs{H} \exp\Big(-\frac{\epsilon\delta \alpha n_1}{|G|^2}\Big)
+ 2\exp\big(-2\rho^2 n_1\big).
\end{align*}
Lemma~\ref{Lem: count codes} tells us that
\[
\abs{H}^{n_1}
-K_G \abs{H}^{n_1}\Big( \exp(-c_{G,\delta}n_1) +2^{-n_1} \Big)
\leq \abs{\SSS{n_1,\delta}{H}{H}}
\leq \abs{H}^{n_1},
\]
and
\[
\abs{H}^{n_2}
-K_G \abs{H}^{n_2} \Big( \exp(-c_{G,\delta}n_2) +2^{-n_2} \Big)
\leq \abs{\SSS{n_2,\delta}{H}{H}}
\leq \abs{H}^{n_2}.
\]
Therefore,
\begin{align*}
&\Big|
\abs{\SSS{n_1,\delta}{H}{H}}
\times\abs{\SSS{n_2,\delta}{H}{H}}
-\abs{H}^{n_1+n_2}
\Big|\\
&\leq K_G \abs{H}^{n_1+n_2} 
\Big(\exp(-c_{G,\delta}n_1)+ 2^{-n_1}+ \exp(-c_{G,\delta}n_2) +2^{-n_2}\Big).
\end{align*}
This implies
\[
\bigg|
\frac{\abs{\SSS{n_1,\delta}{H}{H}}
\times\abs{\SSS{n_2,\delta}{H}{H}}}{\abs{H}^{n_1+n_2}}
-1
\bigg|
\leq 2 K_G \Big(\exp(-c_{G,\delta}\alpha n_1) +2^{-\alpha n_1} \Big).
\]
The result follows.
\end{proof}

\subsection{Error term}

We are now left with the case when $T_1$, $H_1$, $T_2$, $H_2$ are not all the same.
So at least of the following hold:
\begin{itemize}
    \item $T_1\subsetneq H_1$;
    \item $T_2\subsetneq H_2$;
    \item $T_1\neq T_2$.
\end{itemize}
We will actually split this further into the following cases:
\begin{itemize}
    \item $T_1\subsetneq H_1$;
    \item $T_2\subsetneq H_2$;
    \item $T_1\not\subseteq T_2$;
    \item $T_2\not\subseteq T_1$ and $\alpha\neq 1$;
    \item $T_1=H_1$, $T_2=H_2$, $T_1\subsetneq T_2$ and $\alpha=1$.
\end{itemize}

For $t\in\{1,2\}$, denote
\[
e_t
= \begin{cases}
    1-\epsilon &\text{if } T_t\subsetneq H_t\\
    1 &\text{if } T_t= H_t.
\end{cases}
\]
We will consider $F_1\in \SSS{n_1,\delta}{H_1}{T_1}$ and $F_2\in \SSS{n_2,\delta}{H_2}{T_2}$ and bound
\[
\PPP[(F_1,F_2)\circ \Mnot=0 
\text{ and } 
\abs{\gamma(\Mnot) -\tfrac{1}{p} n_1}\leq \rho n_1].
\]
We start by bounding the probability that $(F_1,F_2)\circ \Mnot=0$ and $\Gamma(\Mnot)=S$ for a fixed subset $S\subseteq [n_1]$.

\begin{lemma}\label{Lem: depth prob intermediate}
Suppose $\frac{1}{p}<\alpha\leq 1$ and $0<\rho<\min(\frac{1}{p},\alpha-\frac{1}{p})$.
Fix $S\subseteq [n_1]$.
Consider $F_1\in \SSS{n_1,\delta}{H_1}{T_1}$ and $F_2\in \SSS{n_2,\delta}{H_2}{T_2}$.
Assuming $\alpha n_1\leq n_2\leq n_1$ and that $n_1$ is sufficiently large in terms of $G, \alpha,\delta, \epsilon$, then
\begin{align*}
&\PPP[(F_1,F_2)\circ \Mnot=0 \text{ and }\Gamma(\Mnot)=S]\\
&\leq (\tfrac{1}{p})^{\abs{S}} (\tfrac{1}{\abs{G}})^{n_1-\abs{S}}
\frac{e_1^{n_2} e_2^{\abs{S}}}{\abs{T_1}^{n_2} \abs{T_2}^{\abs{S}}}
\left(1+4\abs{G}n_1 \exp\Big(-\frac{\epsilon \delta \alpha n_1}{|G|^2}\Big) \right)\\
&\quad\quad\times \prod_{i\in [n_1]\setminus S} \sum_{l\in (\Z/\abs{G}\Z)^{\times}} \PPP[l F_2(B_i)= F_1(u_i)].
\end{align*}
\end{lemma}
\begin{proof}
Recall from \eqref{Eqn: Probabilty as product} that
\begin{align*}
&\PPP[(F_1,F_2)\circ \Mnot=0 \text{ and }\Gamma(\Mnot)=S]\\
&=\prod_{i\in S} \PPP[c_iF_1(u_i) +F_2(B_i)=0\text{ and } c_i\equiv 0\pmod{p}]\\
&\quad\quad\times\prod_{i\in [n_1]\setminus S} \PPP[c_iF_1(u_i) +F_2(B_i)=0\text{ and } c_i\not\equiv 0\pmod{p}]\\
&\quad\quad\times\prod_{j\in [n_2]} \PPP[d_jF_2(v_j) +F_1(A_j)=0].
\end{align*}
For each $j\in [n_2]$, we know from Corollary~\ref{Cor: combined bound} that
\[
\PPP[F_1(A_j)= -d_jF_2(v_j)]
\leq e_1\left( \frac{1}{\abs{T_1}} + \exp\Big(-\frac{\epsilon \delta n_1}{|G|^2}\Big)\right).\]
Similarly, for each $i\in S$, we know by Corollary~\ref{Cor: combined bound} that
\begin{align*}
&\PPP[c_iF_1(u_i) +F_2(B_i)=0\text{ and } c_i\equiv 0\pmod{p}]\\
&=\PPP[ F_2(B_i)=- c_iF_1(u_i)\mid c_i\equiv 0\pmod{p}]
\PPP[c_i\equiv 0\pmod{p}]\\
&\leq e_2 \left( \frac{1}{\abs{T_2}} + \exp\Big(-\frac{\epsilon \delta n_2}{|G|^2}\Big)\right) \times \frac{1}{p}.
\end{align*}
Finally, consider $i\in [n_1]\setminus S$. Note that even though $c_i\in \Zp$, the value of $c_i F_1(u_i)$ only depends on $c_i\pmod{\abs{G}}$.
Therefore, we see that
\begin{align*}
&\PPP[c_iF_1(u_i) +F_2(B_i)=0\text{ and } c_i\not\equiv 0\pmod{p}]\\
&=\sum_{k\in (\Z/\abs{G}\Z)^{\times}} \PPP[F_2(B_i)= -kF_1(u_i)] \PPP[c_i\equiv k\pmod{\abs{G}}]\\
&=\frac{1}{\abs{G}} \sum_{l\in (\Z/\abs{G}\Z)^{\times}} \PPP[l F_2(B_i)= F_1(u_i)].
\end{align*}
It follows that
\begin{align*}
&\PPP[(F_1,F_2)\circ \Mnot=0 \text{ and }\Gamma(\Mnot)=S]\\
&\leq (\tfrac{1}{p})^{\abs{S}} (\tfrac{1}{\abs{G}})^{n_1-\abs{S}}
e_1^{n_2}
\left(\frac{1}{\abs{T_1}} + \exp\Big(-\frac{\epsilon \delta n_1}{|G|^2}\Big)\right)^{n_2}
e_2^{\abs{S}}
\left( \frac{1}{\abs{T_2}} + \exp\Big(-\frac{\epsilon \delta n_2}{|G|^2}\Big)\right)^{\abs{S}}\\
&\quad\quad\times \prod_{i\in [n_1]\setminus S} \sum_{l\in (\Z/\abs{G}\Z)^{\times}} \PPP[l F_2(B_i)= F_1(u_i)].
\end{align*}

If $n_1$ is sufficiently large to ensure
\[
\frac{\exp\Big(-\frac{\epsilon \delta \alpha n_1}{|G|^2}\Big)}{2^{\frac{1}{2n_1}}-1}
<\frac{1}{\abs{G}}\leq \min\Big(\frac{1}{\abs{T_1}}, \frac{1}{\abs{T_2}}\Big),
\]
then we are done by Corollary~\ref{Cor: estimate product}.
\end{proof}

Next, we sum over the different subsets $S\subseteq[n_1]$ that satisfy $\abs{\abs{S}-\frac{1}{p}n_1}\leq \rho n_1$ and also over different $F_1\in \SSS{n_1,\delta}{H_1}{T_1}$ and $F_2\in \SSS{n_2,\delta}{H_2}{T_2}$.

For each $F_1\in \Hom(\Zp^{n_1},G)$, there is some $\sigma\subseteq[n_1]$ with $\abs{\sigma}\leq \delta\log_p(D_1) n_1$ such that $F_1(V_{\setminus \sigma})\subseteq T_1$. We will consider fixed subsets $\sigma\subseteq [n_1]$ and sum over those $F_1\in \Hom(\Zp^{n_1},G)$ for which $F_1(V_{\setminus \sigma})\subseteq T_1$. We will denote the sum over such $F_1$ as $\sum_{F_1:\sigma}$.

\begin{proposition}\label{Prop: ind diag error term}
Suppose $\frac{1}{p}<\alpha\leq 1$ and $0<\rho< \min(\alpha-\frac{1}{p},\frac{1}{p})$.
Assume that the subgroups $T_1\subseteq H_1\subseteq G$ and $T_2\subseteq H_2\subseteq G$ satisfy at least one of the following
\begin{itemize}
    \item $T_1\subsetneq H_1$;
    \item $T_2\subsetneq H_2$;
    \item $T_1\not\subseteq T_2$;
    \item $T_2\not\subseteq T_1$ and $\alpha\neq 1$.
\end{itemize}
Assuming $\delta>0$ is sufficiently small in terms of these, we have
\[
\lim_{n\to\infty}
\sum_{F_1\in \SSS{n,\delta}{H_1}{T_1}}
\sum_{F_2\in \SSS{\ceil{\alpha n},\delta}{H_2}{T_2}}
\PPP[(F_1,F_2)\circ \Mna=0 \text{ and } \abs{\gamma(\Mna)-\tfrac{1}{p}n} \leq \rho n]
=0.
\]
\end{proposition}
\begin{proof}
Denote $D_1=[H_1:T_1]$, $D_2=[H_2:T_2]$, $n_1=n$ and $n_2=\ceil{\alpha n}$.
First, consider a fixed $S\subseteq[n_1]$. Then we have
\begin{align*}
&\sum_{F_1\in \SSS{n_1,\delta}{H_1}{T_1}}
\sum_{F_2\in \SSS{n_2,\delta}{H_2}{T_2}}
\PPP[(F_1,F_2)\circ \Mna=0 \text{ and } \Gamma(\Mna)=S]\\
&\leq (\tfrac{1}{p})^{\abs{S}} (\tfrac{1}{\abs{G}})^{n_1-\abs{S}}
\frac{e_1^{n_2} e_2^{\abs{S}}}{\abs{T_1}^{n_2} \abs{T_2}^{\abs{S}}}
\left(1+4\abs{G}n_1 \exp\Big(-\frac{\epsilon \delta \alpha n_1}{|G|^2}\Big) \right)\\
&\quad\quad\times 
\sum_{F_2\in \SSS{n_2,\delta}{H_2}{T_2}}
\sum_{F_1\in \SSS{n_1,\delta}{H_1}{T_1}}
\prod_{i\in [n_1]\setminus S} \sum_{l\in (\Z/\abs{G}\Z)^{\times}} \PPP[l F_2(B_i)= F_1(u_i)].
\end{align*}
We will fix a choice of $F_2$, and bound the sum
\[
\sum_{F_1\in \SSS{n_1,\delta}{H_1}{T_1}}
\prod_{i\in [n_1]\setminus S} \sum_{l\in (\Z/\abs{G}\Z)^{\times}} \PPP[l F_2(B_i)= F_1(u_i)].
\]

We write $V=\Zp^{n_1}$.
For a subset $\sigma\subseteq[n_1]$, denote by $V_{\sigma}$ the subgroup of $V$ generated by $u_i$ for $i\in \sigma$ and $V_{\setminus \sigma}$ to be the subgroup of $V$ generated by $u_i$ for $i\notin \sigma$.
Each $F_1\in \Hom(\Zp^{n_1},G)$ with $\delta$-depth $D_1$ and type $(H_1, T_1)$, has some $\sigma\subseteq[n_1]$ with $\abs{\sigma}= \delta\log_p(D_1) n_1$ such that $F_1(V_{\setminus \sigma})\subseteq T_1$ and $F_1(V_{\sigma})\subseteq H_1$.

We fix a choice of $\sigma\subseteq [n_1]$ (in addition to fixing $F_2$ and $S$).
\begin{itemize}
    \item For $i\in S\cap \sigma$, we have $\abs{H_1}$ choices for $F_1(u_i)$.
    \item For $i\in S\setminus \sigma$, we have $\abs{T_1}$ choices for $F_1(u_i)$.
    \item For $i\in \sigma\setminus S$, $F_1(u_i)\in H_1$, so
\[
\sum_{h\in H_1}
\sum_{l\in (\Z/\abs{G}\Z)^{\times}} \PPP[l F_2(B_i)= h]
=\sum_{l\in (\Z/\abs{G}\Z)^{\times}} \PPP[l F_2(B_i)\in H_1]
\leq \abs{(\Z/\abs{G}\Z)^{\times}}
=(1-\tfrac{1}{p})\abs{G}.
\]

    \item For $i\in (S\cup \sigma)^c$, $F_1(u_i)\in T_1$, so
\[
\sum_{h\in T_1}
\sum_{l\in (\Z/\abs{G}\Z)^{\times}} \PPP[l F_2(B_i)= h]
=\sum_{l\in (\Z/\abs{G}\Z)^{\times}} \PPP[l F_2(B_i)\in T_1]
=(1-\tfrac{1}{p})\abs{G} 
\times \PPP[ F_2(B_i)\in T_1].
\]
    By Lemma~\ref{Lem: bound subgroup}, this is at most
    \[
    (1-\tfrac{1}{p})\abs{G} 
    \times
    \left(\frac{\abs{T_2\cap T_1}}{\abs{T_2}}+\exp\Big(-\frac{\epsilon\delta n_2}{|G|^2}\Big) \right).
    \]
\end{itemize}
We conclude that the sum over those $F_1$ corresponding to $\sigma$ is bounded by
\begin{align*}
&\sum_{F_1:\sigma}
\prod_{i\in [n_1]\setminus S} \sum_{l\in (\Z/\abs{G}\Z)^{\times}} \PPP[l F_2(B_i)= F_1(u_i)]\\
&\leq \abs{H_1}^{\abs{S\cap \sigma}}
\abs{T_1}^{\abs{S\setminus \sigma}}
(1-\tfrac{1}{p})^{n_1-\abs{S}}\abs{G}^{n_1-\abs{S}} 
\left(\frac{\abs{T_2\cap T_1}}{\abs{T_2}}+\exp\Big(-\frac{\epsilon\delta n_2}{|G|^2}\Big) \right)^{n_1-\abs{S\cup \sigma}}\\
&\leq D_1^{\abs{\sigma}}
\abs{T_1}^{\abs{S}}
(1-\tfrac{1}{p})^{n_1-\abs{S}}\abs{G}^{n_1-\abs{S}} 
\left(\frac{\abs{T_2\cap T_1}}{\abs{T_2}}+\exp\Big(-\frac{\epsilon\delta n_2}{|G|^2}\Big) \right)^{n_1-\abs{S}-\abs{\sigma}}\\
&\leq 
\abs{T_1}^{\abs{S}}
(1-\tfrac{1}{p})^{n_1-\abs{S}}\abs{G}^{n_1-\abs{S}} 
\left(\frac{\abs{T_2\cap T_1}}{\abs{T_2}}+\exp\Big(-\frac{\epsilon\delta \alpha n_1}{|G|^2}\Big) \right)^{n_1-\abs{S}}
\left(\frac{D_1 \abs{T_2}}{\abs{T_2\cap T_1}}\right)^{\abs{\sigma}}.
\end{align*}
Once $n_1$ is large enough to ensure,
\[
\frac{\exp\Big(-\frac{\epsilon\delta \alpha n_1}{|G|^2}\Big)}{2^{\frac{1}{n_1}}-1}
<\frac{1}{\abs{G}},
\]
then by Corollary~\ref{Cor: estimate product} we have
\[
\left(\frac{\abs{T_2\cap T_1}}{\abs{T_2}}+\exp\Big(-\frac{\epsilon\delta \alpha n_1}{|G|^2}\Big) \right)^{n_1-\abs{S}}
\leq \frac{\abs{T_2\cap T_1}^{n_1-\abs{S}}}{\abs{T_2}^{n_1-\abs{S}}}
\left(1+2n_1\abs{G}\exp\Big(-\frac{\epsilon\delta \alpha n_1}{|G|^2}\Big)\right).
\]

By summing over $\sigma\subseteq[n_1]$ with $\abs{\sigma}= \delta\log_p(D_1) n_1$, we see that
\begin{align*}
&\sum_{F_1\in \SSS{n_1,\delta}{H_1}{T_1}}
\prod_{i\in [n_1]\setminus S} \sum_{l\in (\Z/\abs{G}\Z)^{\times}} \PPP[l F_2(B_i)= F_1(u_i)]\\
&\leq \binom{n_1}{\delta\log_p(D_1) n_1}
\abs{T_1}^{\abs{S}}
(1-\tfrac{1}{p})^{n_1-\abs{S}}\abs{G}^{n_1-\abs{S}} \\
&\quad\quad\times\frac{\abs{T_2\cap T_1}^{n_1-\abs{S}}}{\abs{T_2}^{n_1-\abs{S}}}
\left(1+2n_1\abs{G}\exp\Big(-\frac{\epsilon\delta \alpha n_1}{|G|^2}\Big)\right)
\left(\frac{D_1 \abs{T_2}}{\abs{T_2\cap T_1}}\right)^{\delta\log_p(D_1) n_1}.
\end{align*}
This implies that
\begin{align*}
&\sum_{F_1\in \SSS{n,\delta}{H_1}{T_1}}
\sum_{F_2\in \SSS{\ceil{\alpha n},\delta}{H_2}{T_2}}
\PPP[(F_1,F_2)\circ \Mna=0 \text{ and } \Gamma(\Mna)=S]\\
&\leq (\tfrac{1}{p})^{\abs{S}} (\tfrac{1}{\abs{G}})^{n_1-\abs{S}}
\frac{e_1^{n_2} e_2^{\abs{S}}}{\abs{T_1}^{n_2} \abs{T_2}^{\abs{S}}}
\left(1+4\abs{G}n_1 \exp\Big(-\frac{\epsilon \delta \alpha n_1}{|G|^2}\Big) \right)\\
&\times
\binom{n_2}{\delta n_2\log_p(D_2)} \abs{T_2}^{n_2} D_2^{\delta n_2\log_p(D_2)}
\times \binom{n_1}{\delta\log_p(D_1) n_1}
\abs{T_1}^{\abs{S}}
(1-\tfrac{1}{p})^{n_1-\abs{S}}\abs{G}^{n_1-\abs{S}} \\
&\times\frac{\abs{T_2\cap T_1}^{n_1-\abs{S}}}{\abs{T_2}^{n_1-\abs{S}}}
\left(1+2n_1\abs{G}\exp\Big(-\frac{\epsilon\delta \alpha n_1}{|G|^2}\Big)\right)
\left(\frac{D_1 \abs{T_2}}{\abs{T_2\cap T_1}}\right)^{\delta\log_p(D_1) n_1}.
\end{align*}
Rearranging the terms and using the fact that $\binom{a}{b}\leq 2^{2\sqrt{ab}}$, we see that the above expression is at most
\begin{align*}
&(\tfrac{1}{p})^{\abs{S}}  (1-\tfrac{1}{p})^{n_1-\abs{S}}
e_1^{n_2} e_2^{\abs{S}}
\left(\frac{\abs{T_2\cap T_1}}{\abs{T_1}}\right)^{n_2-\abs{S}}
\left(\frac{\abs{T_2\cap T_1}}{\abs{T_2}}\right)^{n_1-n_2}\\
&\times
2^{2n_2(\delta\log_p(D_2))^{1/2}}
2^{2n_1(\delta\log_p(D_1))^{1/2}}
D_2^{\delta n_2\log_p(D_2)}
\left(\frac{D_1 \abs{T_2}}{\abs{T_2\cap T_1}}\right)^{\delta\log_p(D_1) n_1}\\
&\times \left(1+4\abs{G}n_1 \exp\Big(-\frac{\epsilon \delta \alpha n_1}{|G|^2}\Big) \right)
\left(1+2n_1\abs{G}\exp\Big(-\frac{\epsilon\delta \alpha n_1}{|G|^2}\Big)\right).
\end{align*}
Since $e_1,e_2\leq 1$, $n_2=\ceil{\alpha n}$ and $(\frac{1}{p}-\rho)n \leq \abs{S}\leq (\frac{1}{p}+\rho)n$ we see that
\begin{align*}
&e_1^{n_2} e_2^{\abs{S}}
\left(\frac{\abs{T_2\cap T_1}}{\abs{T_1}}\right)^{n_2-\abs{S}}
\left(\frac{\abs{T_2\cap T_1}}{\abs{T_2}}\right)^{n_1-n_2}\\
&\leq 
\left(
e_1^{\alpha} e_2^{\frac{1}{p}-\rho}
\left(\frac{\abs{T_2\cap T_1}}{\abs{T_1}}\right)^{\alpha-\frac{1}{p}-\rho}
\left(\frac{\abs{T_2\cap T_1}}{\abs{T_2}}\right)^{1-\alpha}
\right)^n
\frac{\abs{T_2}}{\abs{T_2\cap T_1}}.
\end{align*}
Next, note that
\begin{align*}
&2^{2n_2(\delta\log_p(D_2))^{1/2}}
2^{2n_1(\delta\log_p(D_1))^{1/2}}
D_2^{\delta n_2\log_p(D_2)}
\left(\frac{D_1 \abs{T_2}}{\abs{T_2\cap T_1}}\right)^{\delta\log_p(D_1) n_1}\\
&\leq 2^{2(\log_p(D_2))^{1/2}} D_2^{\log_p(D_2)}
\left(
2^{2\alpha(\log_p(D_2))^{1/2}}
2^{2(\log_p(D_1))^{1/2}}
D_2^{\alpha\log_p(D_2)}
\left(\frac{D_1 \abs{T_2}}{\abs{T_2\cap T_1}}\right)^{\log_p(D_1)}
\right)^{n\sqrt{\delta}}\\
&\leq 2^{2(\log_p(\abs{G}))^{1/2}} \abs{G}^{\log_p(\abs{G})}
\left(
2^{4(\log_p(\abs{G}))^{1/2}}
\abs{G}^{3\log_p(\abs{G})}
\right)^{n\sqrt{\delta}}.
\end{align*}
Further note that for sufficiently large $n$, we have
\[
\left(1+4\abs{G}n_1 \exp\Big(-\frac{\epsilon \delta \alpha n_1}{|G|^2}\Big) \right)
\left(1+2n_1\abs{G}\exp\Big(-\frac{\epsilon\delta \alpha n_1}{|G|^2}\Big)\right)
<2.
\]
We write
\[
B=2^{4(\log_p(\abs{G}))^{1/2}}
\abs{G}^{3\log_p(\abs{G})},\quad\quad
b=\frac{\abs{T_2}}{\abs{T_2\cap T_1}}
2^{2(\log_p(\abs{G}))^{1/2}} \abs{G}^{\log_p(\abs{G})} 
2.
\]
Then we see that
\begin{align*}
&\sum_{F_1\in \SSS{n,\delta}{H_1}{T_1}}
\sum_{F_2\in \SSS{\ceil{\alpha n},\delta}{H_2}{T_2}}
\PPP[(F_1,F_2)\circ \Mna=0 \text{ and } \Gamma(\Mna)=S]\\
&\leq (\tfrac{1}{p})^{\abs{S}}  
(1-\tfrac{1}{p})^{n_1-\abs{S}}
b \left(
e_1^{\alpha} e_2^{\frac{1}{p}-\rho}
\left(\frac{\abs{T_2\cap T_1}}{\abs{T_1}}\right)^{\alpha-\frac{1}{p}-\rho}
\left(\frac{\abs{T_2\cap T_1}}{\abs{T_2}}\right)^{1-\alpha}
B^{\sqrt{\delta}}
\right)^n.
\end{align*}
By summing over different $S$, we see that
\begin{align*}
&\sum_{F_1\in \SSS{n,\delta}{H_1}{T_1}}
\sum_{F_2\in \SSS{\ceil{\alpha n},\delta}{H_2}{T_2}}
\PPP[(F_1,F_2)\circ \Mna=0 \text{ and } \abs{\gamma(\Mna)-\tfrac{1}{p}n} \leq \rho n]\\
&\leq b \left(
e_1^{\alpha} e_2^{\frac{1}{p}-\rho}
\left(\frac{\abs{T_2\cap T_1}}{\abs{T_1}}\right)^{\alpha-\frac{1}{p}-\rho}
\left(\frac{\abs{T_2\cap T_1}}{\abs{T_2}}\right)^{1-\alpha}
B^{\sqrt{\delta}}
\right)^n.
\end{align*}
Note that $e_1$, $e_2$, $\frac{\abs{T_2\cap T_1}}{\abs{T_1}}$ and $\frac{\abs{T_2\cap T_1}}{\abs{T_2}}$ are all at most one.
Now, at least one of the following holds
\begin{itemize}
    \item $T_1\subsetneq H_1$.
    In this case $e_1<1$ and $\alpha>0$, so $e_1^\alpha<1$.
    \item $T_2\subsetneq H_2$.
    In this case $e_2<1$ and $\frac{1}{p}-\rho>0$, so $e_2^{\frac{1}{p}-\rho}<1$.
    \item $T_1\not\subseteq T_2$.
    In this case $\frac{\abs{T_1\cap T_2}}{\abs{T_1}}<1$ and $\alpha-\frac{1}{p}-\rho>0$, so $\Big(\frac{\abs{T_1\cap T_2}}{\abs{T_1}}\Big)^{\alpha-\frac{1}{p}-\rho}<1$.
    \item $T_2\not\subseteq T_1$ and $\alpha\neq 1$.
    In this case $\frac{\abs{T_1\cap T_2}}{\abs{T_2}}<1$ and $1-\alpha>0$, so $\Big(\frac{\abs{T_1\cap T_2}}{\abs{T_2}}\Big)^{1-\alpha}<1$.
\end{itemize}
This means that in all cases
\[
e_1^{\alpha} e_2^{\frac{1}{p}-\rho}
\left(\frac{\abs{T_2\cap T_1}}{\abs{T_1}}\right)^{\alpha-\frac{1}{p}-\rho}
\left(\frac{\abs{T_2\cap T_1}}{\abs{T_2}}\right)^{1-\alpha}<1.
\]
We can therefore choose a sufficiently small $\delta$ to ensure that
\[
e_1^{\alpha} e_2^{\frac{1}{p}-\rho}
\left(\frac{\abs{T_2\cap T_1}}{\abs{T_1}}\right)^{\alpha-\frac{1}{p}-\rho}
\left(\frac{\abs{T_2\cap T_1}}{\abs{T_2}}\right)^{1-\alpha}
B^{\sqrt{\delta}}
<1.
\]
The result follows.
\end{proof}

We are only left with the case $T_1=H_1$, $T_2=H_2$, $T_1\subsetneq T_2$ and $\alpha=1$.

\begin{proposition}\label{Prop: ind diag error term alpha 1}
Suppose $\alpha=1$ and $0<\rho< \min(\alpha-\frac{1}{p},\frac{1}{p})$.
Assume that the subgroups $T_1\subseteq H_1\subseteq G$ and $T_2\subseteq H_2\subseteq G$ satisfy $T_1=H_1$, $T_2=H_2$ and $T_1\subsetneq T_2$.
For $\delta>0$, we have
\[
\lim_{n\to\infty}
\sum_{F_1\in \SSS{n,\delta}{T_1}{T_1}}
\sum_{F_2\in \SSS{n,\delta}{T_2}{T_2}}
\PPP[(F_1,F_2)\circ \Mna=0 \text{ and } \abs{\gamma(\Mna)-\tfrac{1}{p}n} \leq \rho n]
=0.
\]
\end{proposition}
\begin{proof}
Since $\alpha=1$, $n_1=n_2=n$.
We have $e_1=e_2=1$, $F_1\in \SSS{n,\delta}{T_1}{T_1}$ are codes of distance $\delta n$ with image $T_1$ and $F_2\in \SSS{n,\delta}{T_2}{T_2}$ are codes of distance $\delta n$ with image $T_2$.

By the independence of columns of $\Mnot$, we have
\[
\PPP[(F_1,F_2)\circ \Mnot=0]
=\prod_{i\in [n_1]} \PPP[F_2(B_i)=-c_iF_1(u_i)]
\prod_{j\in [n_2]} \PPP[F_1(A_j)=-d_jF_2(v_j)].
\]

For each $i\in [n]$, we know by Corollary~\ref{Cor: combined bound} that
\[
\PPP[F_2(B_i)=-c_iF_1(u_i)]
\leq  \frac{1}{\abs{T_2}} + \exp\Big(-\frac{\epsilon \delta n}{|G|^2}\Big).
\]
For each $j\in [n]$, we know from Corollary~\ref{Cor: combined bound} that
\[
\PPP[F_1(A_j)= -d_jF_2(v_j)]
\leq \frac{1}{\abs{T_1}} + \exp\Big(-\frac{\epsilon \delta n}{|G|^2}\Big).
\]
However, if $F_2(v_j)\in T_2\setminus T_1$, then we can get a tighter upper bound. Since $F_1(A_j)$ can only take values in $T_1$, if $F_1(A_j)= -d_jF_2(v_j)$ then $d_jF_2(v_j)\in T_1$. If $F_2(v_j)\notin T_1$, this forces $d_j\equiv 0\pmod{p}$.
We conclude that if $F_2(v_j)\in T_2\setminus T_1$, then
\begin{align*}
\PPP[F_1(A_j)= -d_jF_2(v_j)]
&= \PPP[F_1(A_j)= -d_jF_2(v_j) \mid d_jF_2(v_j)\in T_1]
\PPP[d_jF_2(v_j)\in T_1]\\
&\leq \left( \frac{1}{\abs{T_1}} + \exp\Big(-\frac{\epsilon \delta n}{|G|^2}\Big)\right)
\PPP[d_j\equiv 0\pmod{p}]\\
&= \frac{1}{p}
\left( \frac{1}{\abs{T_1}} + \exp\Big(-\frac{\epsilon \delta n}{|G|^2}\Big)\right).
\end{align*}
Since $F_2$ is a code of distance $\delta n$ with image $T_2$, we know that 
\[
\abs{\{j\in [n]: F_2(v_j)\notin T_1\}}
\geq \delta n.
\]
It follows that
\[
\PPP[(F_1,F_2)\circ \Mnot=0 ]
\leq 
\left( \frac{1}{\abs{T_2}} + \exp\Big(-\frac{\epsilon \delta n}{|G|^2}\Big)\right)^{n}
\Big(\frac{1}{p}\Big)^{\delta n}
\left(\frac{1}{\abs{T_1}} + \exp\Big(-\frac{\epsilon \delta n}{|G|^2}\Big)\right)^{n}.
\]
Once $n$ is large enough to ensure,
\[
\frac{\exp\Big(-\frac{\epsilon\delta  n}{|G|^2}\Big)}{2^{\frac{1}{2n}}-1}
<\frac{1}{\abs{G}},
\]
then by Corollary~\ref{Cor: estimate product} we know that
\[
\PPP[(F_1,F_2)\circ \Mna=0]
\leq \Big(\frac{1}{p}\Big)^{\delta n}
\frac{1}{\abs{T_2}^n \abs{T_1}^n}
\left( 1 + 4n\abs{G}\exp\Big(-\frac{\epsilon \delta n}{|G|^2}\Big)\right).
\]
Since $\abs{\SSS{n,\delta}{T_1}{T_1}}\leq \abs{T_1}^n$ and $\abs{\SSS{n,\delta}{T_2}{T_2}}\leq \abs{T_2}^n$, we conclude that
\begin{align*}
&\sum_{F_1\in \SSS{n,\delta}{T_1}{T_1}}
\sum_{F_2\in \SSS{n,\delta}{T_2}{T_2}}
\PPP[(F_1,F_2)\circ \Mna=0 \text{ and } \abs{\gamma(\Mna)-\tfrac{1}{p}n} \leq \rho n]\\
&\leq\sum_{F_1\in \SSS{n,\delta}{T_1}{T_1}}
\sum_{F_2\in \SSS{n,\delta}{T_2}{T_2}}
\PPP[(F_1,F_2)\circ \Mna=0]\\
&\leq \Big(\frac{1}{p}\Big)^{\delta n}
\left( 1 + 4n\abs{G}\exp\Big(-\frac{\epsilon \delta n}{|G|^2}\Big)\right).
\end{align*}
The result follows.
\end{proof}

Notice that Proposition~\ref{Prop: ind diag main term}, Proposition~\ref{Prop: ind diag error term} and Proposition~\ref{Prop: ind diag error term alpha 1} together prove Proposition~\ref{Prop: ind diagonal pieces}. We have seen that Proposition~\ref{Prop: ind diagonal pieces} implies Proposition~\ref{Prop: ind diag hom}, which implies Proposition~\ref{Prop: ind diag sur}, which in turn implies Theorem~\ref{Thm: distribution ind diag}.

\section{Setup for Laplacian model}\label{Sec: set laplacian dir}

We continue to have a fixed finite abelian $p$-group $G$, and we wish to estimate the expected number of surjections from $G(\Lnot)$ to $G$.

We start by turning the problem of computing the moments of the Laplacian model, into a sum of certain probabilities of the independent diagonal model. This is very similar to the argument of Wood in \cite{wood2017distribution}.

We denote $\Sur^*(\Zpnot,G)$ to be the set of those maps in $\Sur(\Zpnot,G)$, that are still surjective when restricted to $\Znot$. Note that every map in $\Sur(\Znot,G)$, can be extended to $|G|$ maps in $\Sur^*(\Zpnot,G)$.
Since $\Lnot\in \Hom(\Zpnot,\Znot)$, we see that
\begin{align*}
&\Erhop[\abs{\Sur(G(\Lnot),G)}]\\
&=\sum_{F\in\Sur(\Znot,G)} \PPP[F\circ \Lnot=0\in \Hom(\Zpnot,G) \text{ and }\abs{\gamma(\Lnot) -\tfrac{1}{p} n_1}\leq \rho n_1]\\
&=\frac{1}{\abs{G}} \sum_{F\in\Sur^*(\Zpnot,G)} \PPP[F\circ \Lnot=0\in \Hom(\Zpnot,G) \text{ and }\abs{\gamma(\Lnot) -\tfrac{1}{p} n_1}\leq \rho n_1].
\end{align*}
We will relate this to our computations with $\Mnot$.
Recall that $\Mnot$ includes $c_j$ and $d_i$ as its entries, but these do not appear in $\Lnot$. Therefore, the $c_j$ and $d_i$ are independent of the entries of $\Lnot$.
Moreover, they are Haar-uniform, so we see that
\begin{align*}
&\PPP[F\circ \Lnot=0\in \Hom(\Zpnot,G) \text{ and }\abs{\gamma(\Lnot) -\tfrac{1}{p} n_1}\leq \rho n_1]\\
&= \PPP[F\circ \Lnot=0\in \Hom(\Zpnot,G) \text{ and }\abs{\gamma(\Lnot) -\tfrac{1}{p} n_1}\leq \rho n_1]\\
&\times |G|^{n_1+n_2}
\prod_{i=1}^{n_1} \PPP[c_i\equiv -\sum_j b_{ij}\pmod{\abs{G}}] 
\prod_{j=1}^{n_2} \PPP[d_j\equiv -\sum_i a_{ij}\pmod{\abs{G}}]\\
&=|G|^{n_1+n_2}
\PPP\Big[F\circ \Lnot=0\in \Hom(\Zpnot,G) \text{ and }\abs{\gamma(\Lnot) -\tfrac{1}{p} n_1}\leq \rho n_1\\ 
&\quad\quad\quad\quad\quad\quad
\text{ and } d_j\equiv -\sum_i a_{ij}\pmod{\abs{G}}
\text{ and } c_i\equiv -\sum_j b_{ij}\pmod{\abs{G}}\Big]\\
&=|G|^{n_1+n_2}
\PPP\Big[F\circ \Mnot=0\in \Hom(\Zpnot,G) \text{ and }\abs{\gamma(\Mnot) -\tfrac{1}{p} n_1}\leq \rho n_1\\ 
&\quad\quad\quad\quad\quad\quad
\text{ and } d_j +\sum_i a_{ij} \equiv 0 \pmod{\abs{G}}
\text{ and } c_i +\sum_j b_{ij} \equiv 0\pmod{\abs{G}}\Big].
\end{align*}
Denote $F_0\in \Hom(\Zpnot,\ZG)$, where $F_0$ maps a vector to the sum of its entries mod $|G|$. This means that
\begin{align*}
&\Erhop[\abs{\Sur(G(\Lnot),G)}]\\
&\!\!\!=\abs{G}^{n_1+n_2-1}\!\!\!\!\!\!\!\!\!\!\!\!\!\!\! \sum_{F\in\Sur^*(\Zpnot,G)}\!\!\!\!\!\!\!\!\!\!\!\!
\PPP\Big[(F,F_0)\circ \Mnot=0\in \Hom(\Zpnot,G \oplus\ZG),\abs{\gamma(\Mnot) -\tfrac{1}{p} n_1}\leq \rho n_1\Big].
\end{align*}
Note that $F\in \Sur^*(\Zpnot,G)$ if and only if $(F,F_0)\in \Sur(\Zpnot,G\oplus (\ZG))$.
Denote the projections from $G\oplus (\ZG)$ to $G$ as $\pi_1$, and to $\ZG$ as $\pi_2$.
Denote
\[
\Sur^1(\Zpnot,G\oplus(\ZG))
:=\{F\in \Sur(\Zpnot,G\oplus(\ZG)) : \pi_2\circ F=F_0\}.
\]
We conclude that
\begin{equation}\label{Eqn: Exp surj from laplacian}
\begin{split}
&\Erhop[\abs{\Sur(G(\Lnot),G)}]\\
&=\abs{G}^{n_1+n_2-1}\!\!\!\!\!\!\!\!\!\!\!\! \sum_{F\in\Sur^1(\Zpnot,G\oplus (\ZG))}\!\!\!\!\!\!\!\!\!\!\!\!
\PPP\Big[F\circ \Mnot=0 \text{ and }\abs{\gamma(\Mnot) -\tfrac{1}{p} n_1}\leq \rho n_1\Big].
\end{split}
\end{equation}
Note that the sum in \eqref{Eqn: Exp surj from laplacian} is closely related to the surjective $(G\oplus \ZG)$-moment of $\Coker(\Mnot)$.
The difference is that we are only summing over some of the surjections from $\Zpnot$ to $G\oplus \ZG$.
We will use results from the previous section to estimate the probability that $F\circ \Mnot=0$ and $\abs{\gamma(\Mnot) -\tfrac{1}{p} n_1}\leq \rho n_1$.

Also denote
\[
\Hom^1(\Zpn,G\oplus(\ZG))
:=\{F\in \Hom(\Zpn,G\oplus(\ZG)) : \pi_2\circ F=F_0\}.
\]

For $F\in \Sur^1(\Zpnot,G\oplus(\ZG))$, let $F_1$ be the restriction of $F$ to $\Zp^{n_1}$ and $F_2$ be its restriction to $\Zp^{n_2}$. Then $F_1\in \Hom^1(\Zp^{n_1},G\oplus \ZG)$, $F_2\in \Hom^1(\Zp^{n_2},G\oplus \ZG)$ and $F_1(\Zp^{n_1}) +F_2(\Zp^{n_2})=G\oplus \ZG$.
This means that if $F_1$ is of type $(H_1,T_1)$ and $F_2$ is of type $(H_2,T_2)$, then $H_1+H_2=G\oplus \ZG$.
Moreover, since $F$ maps each basis vector of $\Zpn$ to something of the form $(g,1)$, it follows that $\pi_2$ is surjective when restricted to $H_1$, $T_1$, $H_2$ or $T_2$. That is
\[
\pi_2(H_1) =\pi_2(T_1) =
\pi_2(H_2) =\pi_2(T_2) =\ZG.
\]

\begin{definition}
We call a subgroup $H$ of $G\oplus \ZG$, $\pi_2$-full if $\pi_2(H)=\ZG$.
\end{definition}
Given $\pi_2$-full subgroups $T\subseteq H\subseteq G\oplus \ZG$, we denote
\[
\SSO{n,\delta}{H}{T}
:=\{F\in \Hom^1(\Zp^n,G\oplus\ZG): \delta-\text{type }(H,T)\}.
\]

We will estimate the size of $\SSO{n,\delta}{H}{T}$. We start with the case $T\subsetneq H$.
Note that the size of $\SSS{n,\delta}{H}{T}$ was of the order of $\abs{T}^n$, but the size of $\SSO{n,\delta}{H}{T}$ should be smaller because each basis element has to be mapped to something of the form $(g,1)$.

\begin{lemma}
Consider $D>1$ and subgroups $T\subsetneq H\subseteq G\oplus\ZG$ such that $H$, $T$ are both $\pi_2$-full and $[H:T]=D$. Then we have
\[
\abs{\SSO{n,\delta}{H}{T}}
\leq 2^{n\sqrt{2\delta\log_p(\abs{G})}}
 \frac{\abs{T}^n}{\abs{G}^n} D^{\delta\log_p(D)n}.
\]
\end{lemma}
\begin{proof}
Any $F\in\SSO{n,\delta}{H}{T}$ would have a $\sigma\subseteq [n]$, such that $\abs{\sigma}<\delta\log_p(D)n$ and $F(\Vsigma)\subseteq T$, $F(V_{\sigma})\subseteq H$. By enlarging $\sigma$, we can assume that $\abs{\sigma}=\ceil{\delta\log_p(D) n}-1$.

Now there are $\binom{n}{\ceil{\delta\log_p(D) n}-1}$ choices for $\sigma$ and once we choose $\sigma$ then:
\begin{itemize}
    \item For $i\in \sigma$, $F(v_i)\in H\cap\pi_2^{-1}(1)$. Since $H$ is $\pi_2$-full $\abs{H\cap\pi_2^{-1}(1)}=\frac{\abs{H}}{\abs{G}}=D\frac{\abs{T}}{\abs{G}}$, so $F(v_i)$ has $D\frac{\abs{T}}{\abs{G}}$ choices.
    \item For $i\not\in \sigma$, $F(v_i)\in T\cap\pi_2^{-1}(1)$. Since $T$ is also $\pi_2$-full $\abs{T\cap\pi_2^{-1}(1)}=\frac{\abs{T}}{\abs{G}}$, so $F(v_i)$ has $\frac{\abs{T}}{\abs{G}}$ choices.
\end{itemize}
The result follows since $\binom{n}{\ceil{\delta\log_p(D) n}-1}\leq 2^{n\sqrt{\delta\log_p(\abs{G}^2)}}$.
\end{proof}

For ease of notation we denote $L_G=K_{G\oplus\ZG}$, the number of subgroups of $G\oplus\ZG$. We will now estimate the size of $\SSO{n,\delta}{H}{H}$ for $H=G\oplus\ZG$. We do not need this estimate for other $H$.

\begin{lemma}\label{Lem: count codes H^1}
Consider $0<\delta<\frac{1}{6\log_p(\abs{G})}$.
There is a constant $c_{G,\delta}>0$ (depending on $G$ and $\delta$), such that
\[
\abs{G}^n \Big(1- L_G \exp(-c_{G,\delta}n) - L_G 2^{-n}\Big)
\leq \abs{\SSO{n,\delta}{G\oplus\ZG}{G\oplus\ZG}}
\leq \abs{G}^n.
\]
\end{lemma}
\begin{proof}
The upper bound follows from the fact that
\[
\SSO{n,\delta}{G\oplus\ZG}{G\oplus\ZG} \subseteq \Hom^1(\Zp^n, G\oplus\ZG).
\]
Moreover,
\begin{align*}
&\Hom^1(\Zp^n, G\oplus\ZG) \setminus \SSO{n,\delta}{G\oplus\ZG}{G\oplus\ZG}\\
&=\bigcup_{T\subsetneq G\oplus\ZG} \SSO{n,\delta}{G\oplus\ZG}{T} \cup \Hom^1(\Zp^n,T).
\end{align*}
The number of $T$ is at most $L_G$, and 
\[
\abs{\Hom^1(\Zp^n,T)}
= \frac{\abs{T}^n}{\abs{G}^n} \leq \frac{\abs{G}^{2n}}{2^n\abs{G}^n}.
\]
Denote $[G\oplus\ZG:T]=D>1$. Then we know that
\[
\abs{\SSO{n,\delta}{G\oplus\ZG}{T}}
\leq 2^{n\sqrt{2\delta\log_p(\abs{G})}}
\frac{\abs{G}^{2n}}{\abs{G}^n} D^{-n+\delta\log_p(D)n}
\leq 
\abs{G}^n \Big(2^{\sqrt{2\delta\log_p(\abs{G})}-1+\delta\log_p(\abs{G}^2)}\Big)^n.
\]
The bound on $\delta$ also ensures that $\sqrt{2\delta\log_p(\abs{G})}-1+2\delta\log_p(\abs{G})<0$. The result follows.
\end{proof}

Now consider $F\in\SSO{n,\delta}{H}{T}$, for $\pi_2$-full subgroups $T\subseteq H\subseteq G\oplus\ZG$. Let $X\in\Zp^n$ be a random vector whose entries are independent and $\epsilon$-balanced.
\begin{itemize}
    \item If $T=H$, then Lemma~\ref{Lem: Wood estimate code} implies that for every $f\in H$, we have
    \[
    \left| \PPP[F(X)=f]-\frac{1}{\abs{H}}\right|
    \leq \exp\Big(-\frac{\epsilon \delta n}{|G|^4}\Big).
    \]
    \item Denote
    \[
    e=\begin{cases}
    1&\text{if } T=H\\
    1-\epsilon &\text{if }T\subsetneq H.
    \end{cases}
    \]
    Then by Corollary~\ref{Cor: combined bound} for every $f\in G\oplus\ZG$, we have
    \[
    \PPP[F(X)=f]
    \leq e 
     \left(\frac{1}{\abs{T}}+\exp\Big(-\frac{\epsilon\delta n}{|G|^4}\Big) \right).
    \]
    \item In Lemma~\ref{Lem: bound subgroup} we bounded the probability that $F(X)$ belongs to another subgroup.
    However, for the Laplacian model, we will need to bound the probability that $F(X)$ belongs to another subgroup and $\pi_2(F(X))\in (\ZG)^{\times}$.
\end{itemize}

We first bound the probability that $F(X)$ belongs to another subgroup and $\pi_2(F(X))\in (\ZG)^{\times}$, for $F\in\SSO{n,\delta}{H}{T}$ with $T=H$.

\begin{lemma}\label{Lem: bound subgroup, pi_2 full, code}
Suppose $\delta>0$ and $T, T'\subseteq G\oplus\ZG$ are $\pi_2$-full subgroups.
Let $X\in\Zp^n$ be a random vector whose entries are independent and $\epsilon$-balanced.
Then, for every $F\in\SSO{n,\delta}{T}{T}$ we have
\[
\PPP[F(X)\in T' \text{ and } \pi_2(F(X))\in (\ZG)^{\times}]
\leq 
 \left( (1-\tfrac{1}{p}) \frac{\abs{T\cap T'}}{\abs{T}}
 +\abs{G}^2\exp\Big(-\frac{\epsilon\delta n}{|G|^4}\Big) \right).
\]
\end{lemma}
\begin{proof}
We know that $T$ and $T'$ are all $\pi_2$-full.
We divide the proof into two cases based on whether $T\cap T'$ is $\pi_2$-full or not.

\begin{enumerate}
    \item Case 1: $T\cap T'$ is not $\pi_2$-full. In this case, if $F(X)\in T'$, then $F(X)\in T\cap T'$, so $\pi_2(F(X))\notin (\ZG)^{\times}$. Therefore,
    \[
    \PPP[F(X)\in T' \text{ and } \pi_2(F(X))\in (\ZG)^{\times}]
    =0.
    \]
    The result follows.

    \item Case 2: $T\cap T'$ is $\pi_2$-full. Then we know that
    \[
    \abs{(T\cap T')\cap \pi_2^{-1}((\ZG)^{\times}) }
    = (1-\tfrac{1}{p}) \abs{(T\cap T')}.
    \]
    Therefore, Lemma~\ref{Lem: Wood estimate code} says that
    \begin{align*}
    \PPP[F(X)\in T' \text{ and } \pi_2(F(X))\in (\ZG)^{\times}]
    &=\PPP[F (X) \in (T\cap T')\cap \pi_2^{-1}((\ZG)^{\times})]\\
    &\leq (1-\tfrac{1}{p}) \abs{T\cap T'} \left(\frac{1}{\abs{T}}
    +\exp\Big(-\frac{\epsilon \delta n}{|G|^4}\Big) \right).
    \end{align*}
    The result follows.\qedhere
\end{enumerate}
\end{proof}

Next, we bound the probability that $F(X)$ belongs to another subgroup and $\pi_2(F(X))\in (\ZG)^{\times}$, for $F\in\SSO{n,\delta}{H}{T}$ without the requirement that $T=H$. This leads to a slightly weaker bound.

\begin{lemma}\label{Lem: bound subgroup, pi_2 full, depth}
Suppose $\delta>0$, $T\subseteq H\subseteq G\oplus\ZG$ are $\pi_2$-full subgroups, and $T'\subseteq G\oplus \ZG$, is another $\pi_2$-full subgroup.
Let $X\in\Zp^n$ be a random vector whose entries are independent and $\epsilon$-balanced.
Then, for every $F\in\SSO{n,\delta}{H}{T}$
\[
\PPP[F(X)\in T' \text{ and } \pi_2(F(X))\in (\ZG)^{\times}]
\leq 
 \left( (1-\tfrac{1}{p}) \frac{\min(\abs{T},\abs{T'})}{\abs{T}}
 +\abs{G}^2\exp\Big(-\frac{\epsilon\delta n}{|G|^4}\Big) \right).
\]
\end{lemma}
\begin{proof}
Denote $D=[H:T]$.
We know that $H$, $T$ and $T'$ are all $\pi_2$-full.
We divide the proof into two cases based on whether $T\cap T'$ is $\pi_2$-full or not.

\begin{enumerate}
    \item Case 1: $T\cap T'$ is not $\pi_2$-full. Since $T$ and $T'$ are both $\pi_2$-full, this implies that $T\cap T'$ is a proper subset of both $T$ and of $T'$. This means that
    \[
    \abs{T\cap T'}
    \leq \tfrac{1}{p} \min(\abs{T}, \abs{T'})
    \leq (1-\tfrac{1}{p}) \min(\abs{T}, \abs{T'}).
    \]
    Moreover, Lemma~\ref{Lem: bound subgroup} implies that
    \[
    \PPP[F(X)\in T' \text{ and } \pi_2(F(X))\in (\ZG)^{\times}]
    \leq \PPP[F(X)\in T']
    \leq 
     \left(\frac{\abs{T\cap T'}}{\abs{T}}+\exp\Big(-\frac{\epsilon\delta n}{|G|^4}\Big) \right).
    \]
    The result follows.

    \item Case 2: $T\cap T'$ is $\pi_2$-full.
    There is some $\sigma\subseteq [n]$ with $\abs{\sigma}\leq \delta n \log_p(D)$ such that $F(V_{\setminus\sigma})=T$.
    Recall that $V_{\setminus\sigma}$ is the submodule of $\Zp^n$ generated by those basis vectors $u_i$ for which $i\notin \sigma$. Let $V_{\sigma}$ be the submodule generated by $u_i$ with $i\in \sigma$. Now,
    \[
    F\in\Hom(V_{\setminus\sigma}\oplus V_{\sigma},G\oplus\ZG)
    = \Hom(V_{\setminus\sigma},G\oplus\ZG) \oplus \Hom(V_{\sigma},G\oplus \ZG),
    \]
    so we can decompose $F=(F_{\setminus\sigma}, F_{\sigma})$ and $X=(X_{\setminus\sigma}, X_{\sigma})$. We know that
    \[ 
    F(X) 
    =F_{\setminus\sigma}(X_{\setminus\sigma})
    +F_{\sigma}(X_{\sigma}).
    \]
    Thus $F(X)\in T'$ if and only if $F_{\setminus\sigma}(X_{\setminus\sigma}) \equiv -F_{\sigma}(X_{\sigma}) \pmod{T'}$.
    Similarly, $ \pi_2(F(X))\in (\ZG)^{\times}$ if and only if $\pi_2(F_{\setminus\sigma}(X_{\setminus\sigma})) \not\equiv -\pi_2( F_{\sigma}(X_{\sigma})) \pmod{p}$.
    By conditioning based on the value of $F_{\sigma}(X_{\sigma})$, we see that
    \begin{align*}
    &\PPP[F(X)\in T' \text{ and } \pi_2(F(X))\in (\ZG)^{\times}]\\
    &=\sum_{h_0\in H}
    \PPP[F_{\sigma}(X_{\sigma}) =h_0]
    \PPP[F_{\setminus\sigma}(X_{\setminus\sigma}) \equiv -h_0\pmod{ T'} \text{ and } \pi_2(F_{\setminus\sigma}(X_{\setminus\sigma}))\not\equiv -\pi_2(h_0)\pmod{p}].
    \end{align*}
    For $h_0\in H$, define
    \begin{align*}
    S_1(h_0)&=\{h\in T : h\equiv -h_0\pmod{T'}\},\\
    S_2(h_0)&=\{h\in T : h\equiv -h_0\pmod{T'}\text{ and } \pi_2(h) \not \equiv -\pi_2(h_0)\pmod{p}\}.
    \end{align*}
    This means that,
    \begin{align*}
    &\PPP[F(X)\in T' \text{ and } \pi_2(F(X))\in (\ZG)^{\times}]\\
    &=\sum_{h_0\in H}
    \PPP[F_{\sigma}(X_{\sigma}) =h_0]
    \PPP[F_{\setminus\sigma}(X_{\setminus\sigma}) \in S_2(h_0)].
    \end{align*}
    
    We will show that for each $h_0$, $\abs{S_2(h_0)}\leq (1-\frac{1}{p}) \abs{T\cap T'}$.
    \begin{itemize}
        \item If $S_1(h_0)=\emptyset$, then this is trivially true.
        \item If $S_1(h_0)\neq\emptyset$, then pick $h_1\in S_1(h_0)$. It follows that $S_1(h_0)=h_1+T\cap T'$. Since $\pi_2(T\cap T')= \ZG$ (this is because we are in Case 2), it follows that $\abs{S_2(h_0)}\leq (1-\frac{1}{p}) \abs{T\cap T'}$.
    \end{itemize}

    By an identical argument to the one in Lemma~\ref{Lem: bound subgroup}, we see that $F_{\setminus\sigma}$ is a code of distance $\delta n$ with image $T$.
    Therefore, Lemma~\ref{Lem: Wood estimate code} says that
    \[
    \PPP[F_{\setminus\sigma} (X_{\setminus\sigma}) \in S_2(h_0)]
    \leq \abs{S_2(h_0)} \left(\frac{1}{\abs{T}}
    +\exp\Big(-\frac{\epsilon \delta n}{|G|^4}\Big) \right).
    \]
    The result follows since $\abs{S_2(h_0)}\leq (1-\frac{1}{p}) \abs{T\cap T'}$.
\end{enumerate}
\end{proof}

If we take $T'=G\oplus\ZG$, then $F(X)=T'$ is automatic. We are therefore left with the probability that $\pi_2(F(X))\in (\ZG)^{\times}$. We note this in a corollary.

\begin{corollary}\label{Cor: prob pi_2 is unit}
Suppose $\delta>0$, $T\subseteq H\subseteq G\oplus\ZG$ are $\pi_2$-full subgroups.
Let $X\in\Zp^n$ be a random vector whose entries are independent and $\epsilon$-balanced.
Then, for every $F\in\SSO{n,\delta}{H}{T}$
\[
\PPP[\pi_2(F(X))\in (\ZG)^{\times}]
\leq 
 \left( (1-\tfrac{1}{p})
 +\abs{G}^2\exp\Big(-\frac{\epsilon\delta n}{|G|^4}\Big) \right).
\]
\end{corollary}
\begin{proof}
Take $T'=G\oplus\ZG$.
\end{proof}

\section{Laplacian model}\label{Sec: laplacian dir}

Recall the following:
\begin{enumerate}
    \item We had shown in \eqref{Eqn: Exp surj from laplacian} that
    \begin{align*}
    &\Erhop[\abs{\Sur(G(\Lnot),G)}]\\
    &=\abs{G}^{n_1+n_2-1}\!\!\!\!\!\!\!\!\!\!\!\! \sum_{F\in\Sur^1(\Zpnot,G\oplus (\ZG))}\!\!\!\!\!\!\!\!\!\!\!\!
    \PPP\Big[F\circ \Mnot=0 \text{ and }\abs{\gamma(\Mnot) -\tfrac{1}{p} n_1}\leq \rho n_1\Big].
    \end{align*}
    \item Each $F\in\Sur^1(\Zpnot,G\oplus (\ZG))$, can be decomposed as $(F_1, F_2)$, with $F_1\in \Hom^1(\Zp^{n_1},G\oplus\ZG)$, $F_2\in \Hom^1(\Zp^{n_2},G\oplus \ZG)$ satisfying $\im(F_1)+\im(F_2)=G\oplus\ZG$.
    \item Given $\delta>0$, the set of $(F_1,F_2)$ appearing this way can be partitioned as follows:
    Consider $H_1$ and $H_2$ to be $\pi_2$-full subgroups of $G\oplus\ZG$ such that $H_1+H_2=G\oplus\ZG$. Let $T_1$ be a $\pi_2$-full subgroup of $H_1$ and $T_2$ be a $\pi_2$-full subgroup of $H_2$. Then we consider $F_1\in \SSO{n_1,\delta}{H_1}{T_1}$ and $F_2\in \SSO{n_2,\delta}{H_2}{T_2}$.
    On each such piece, our goal is to estimate
\[
\PPP\Big[(F_1,F_2)\circ \Mnot=0 \text{ and }\abs{\gamma(\Mnot) -\tfrac{1}{p} n_1}\leq \rho n_1\Big].
\]
\end{enumerate}

We will show the following.
\begin{proposition}\label{Prop: Laplacian pieces}
Given $\frac{1}{p}<\alpha\leq 1$ and $0<\rho<\min(\frac{1}{p},\alpha-\frac{1}{p})$ and a finite abelian $p$-group $G$, there is $\delta>0$, such that for any $\pi_2$-full subgroups $T_1\subseteq H_1\subseteq G\oplus \ZG$, $T_2\subseteq H_2\subseteq G\oplus\ZG$ with $H_1+H_2=G\oplus\ZG$, we have
\begin{align*}
&\lim_{n\to\infty}
\abs{G}^{n+\ceil{\alpha n}-1}
\sum_{F_1\in \SSO{n,\delta}{H_1}{T_1}}
\sum_{F_2\in \SSO{\ceil{\alpha n},\delta}{H_2}{T_2}}
\PPP[(F_1,F_2)\circ \Mna=0 \text{ and } \abs{\gamma(\Mna)-\tfrac{1}{p}n} \leq \rho n]\\
&\quad\quad\quad=\begin{cases}
\frac{1}{\abs{G}} &\text{if } T_1=H_1=T_2=H_2=G\oplus\ZG\\
0 &\text{otherwise}.
\end{cases}
\end{align*}
\end{proposition}
This will imply Proposition~\ref{Prop: Sur moment Laplacian}.

\begin{customprop}{\ref{Prop: Sur moment Laplacian}}
Given $\frac{1}{p}<\alpha\leq 1$, $0<\rho<\min(\frac{1}{p},\alpha-\frac{1}{p})$ and a finite abelian $p$-group $G$, we have
\[
\lim_{n\to\infty} \Erho[\abs{\Sur(G(\Lna),G)}] 
=\frac{1}{\abs{G}}.
\]
\end{customprop}
\begin{proof}[Proof of Proposition~\ref{Prop: Sur moment Laplacian} assuming Proposition~\ref{Prop: Laplacian pieces}]
Proposition~\ref{Prop: Laplacian pieces} implies that
\[
\lim_{n\to\infty}
\Erhop[\abs{\Sur(G(\Lna),G)}] =\frac{1}{\abs{G}}.
\]
The result follows from Corollary~\ref{Cor: prob of cond to 1}.
\end{proof}

\subsection{Main term}
In this subsection, we consider the case
\[
H_1=H_2=T_1=T_2=G \oplus\ZG.
\]
These will contribute to the main term. In this case $F_1$ and $F_2$ are codes of distance $\delta n_1$ and $\delta n_2$ respectively with image $G\oplus\ZG$, we can therefore use Corollary~\ref{Cor: both codes probability}.

\begin{lemma}
Suppose $\frac{1}{p}<\alpha\leq 1$, $0<\rho<\min(\frac{1}{p},\alpha-\frac{1}{p})$ and $0<\delta<\frac{1}{6\log_p(\abs{G})}$ and $L_{G}$, $c_{G,\delta}$ be the constants from Lemma~\ref{Lem: count codes H^1}.
Suppose $\alpha n_1\leq n_2\leq n_1$.
Denote $S_1=\SSO{n_1,\delta}{G\oplus\ZG}{G\oplus\ZG}$ and $S_2=\SSO{n_2,\delta}{G\oplus\ZG}{G\oplus\ZG}$.
Then, for sufficiently large values of $n_1$ (depending on $G,\epsilon,\delta,\alpha$), we have
\begin{align*}
&\left|
\abs{G}^{n_1+n_2-1}\left(
\sum_{F_1\in S_1}
\sum_{F_2\in S_2}
\PPP[(F_1,F_2)\circ \Mnot=0 \text{ and }
\abs{\gamma(\Mnot) -\tfrac{1}{p} n_1}\leq \rho n_1] 
\right)
- \frac{1}{\abs{G}}
\right|\\
&\leq 4n_1\abs{G} \exp\Big(-\frac{\epsilon\delta \alpha n_1}{|G|^4}\Big)
+ \frac{2}{\abs{G}}\exp\big(-2\rho^2 n_1\big)
+\frac{2 L_G}{\abs{G}} \exp(-c_{G,\delta}\alpha n_1)
+\frac{2 L_G}{\abs{G}} 2^{-\alpha n_1}.
\end{align*}
\end{lemma}
\begin{proof}
We know from Corollary~\ref{Cor: both codes probability} that for sufficiently large values of $n_1$, for each $F_1\in S_1$ and $F_2\in S_2$, we have
\begin{align*}
&\left|
\PPP[(F_1,F_2)\circ \Mnot=0 \text{ and }
\abs{\gamma(\Mnot)- \tfrac{1}{p} n_1}\leq \rho n_1] 
- \frac{1}{\abs{G}^{2(n_1+n_2)}}
\right|\\
&\leq 4n_1\frac{1}{\abs{G}^{2(n_1+n_2-1)}} \exp\Big(-\frac{\epsilon\delta \alpha n_1}{|G|^4}\Big)
+ \frac{2}{\abs{G}^{2(n_1+n_2)}} \exp\big(-2\rho^2 n_1\big).
\end{align*}
We sum over these $F_1$ and $F_2$. 
As an upper bound we use $\abs{S_1}\leq \abs{\Hom^1(\Zp^{n_1},G\oplus\ZG)}=\abs{G}^{n_1}$, $\abs{S_2}\leq \abs{\Hom^1(\Zp^{n_2},G\oplus\ZG)}=\abs{G}^{n_2}$.
\begin{align*}
&\Bigg|
\left(
\sum_{F_1\in S_1}
\sum_{F_2\in S_2}
\PPP[(F_1,F_2)\circ \Mnot=0 \text{ and }
\abs{\gamma(\Mnot) -\tfrac{1}{p} n_1}\leq \rho n_1] 
\right)
- \frac{\abs{S_1}\abs{S_2}}{\abs{G}^{2(n_1+n_2)}}
\Bigg|\\
&\leq 4n_1\frac{1}{\abs{G}^{n_1+n_2-2}}  \exp\Big(-\frac{\epsilon\delta \alpha n_1}{|G|^4}\Big)
+ 2\frac{1}{\abs{G}^{n_1+n_2}} \exp\big(-2\rho^2 n_1\big).
\end{align*}
By Lemma~\ref{Lem: count codes H^1} we see that
\[
\abs{G}^{n_1}
-L_G \abs{G}^{n_1}\Big( \exp(-c_{G,\delta}n_1) +2^{-n_1} \Big)
\leq \abs{S_1}
\leq \abs{G}^{n_1},
\]
and
\[
\abs{G}^{n_2}
-L_G \abs{G}^{n_2} \Big( \exp(-c_{G,\delta}n_2) +2^{-n_2} \Big)
\leq \abs{S_2}
\leq \abs{G}^{n_2}.
\]
Therefore,
\[
\Big|
\abs{S_1}
\abs{S_2}
-\abs{G}^{n_1+n_2}
\Big|\leq L_G \abs{G}^{n_1+n_2} 
\Big(\exp(-c_{G,\delta}n_1)+ 2^{-n_1}+ \exp(-c_{G,\delta}n_2) +2^{-n_2}\Big).
\]
This implies
\[
\bigg|
\frac{\abs{S_1}
\abs{S_2}}{\abs{G}^{n_1+n_2}}
-1
\bigg|
\leq 2 L_G \Big(\exp(-c_{G,\delta}\alpha n_1) +2^{-\alpha n_1} \Big).
\]
We conclude that
\begin{align*}
&\left|
\abs{G}^{n_1+n_2-1}\left(
\sum_{F_1\in S_1}
\sum_{F_2\in S_2}
\PPP[(F_1,F_2)\circ \Mnot=0 \text{ and }
\abs{\gamma(\Mnot) -\tfrac{1}{p} n_1}\leq \rho n_1] 
\right)
- \frac{1}{\abs{G}}
\right|\\
&\leq \abs{G}^{n_1+n_2-1}
\Bigg|
\left(
\sum_{F_1\in S_1}
\sum_{F_2\in S_2}
\PPP[(F_1,F_2)\circ \Mnot=0 \text{ and }
\abs{\gamma(\Mnot) -\tfrac{1}{p} n_1}\leq \rho n_1] 
\right)
- \frac{\abs{S_1}\abs{S_2}}{\abs{G}^{2(n_1+n_2)}}
\Bigg|\\
&\quad\quad+\frac{1}{\abs{G}}
\bigg|
\frac{\abs{S_1}
\abs{S_2}}{\abs{G}^{n_1+n_2}}
-1
\bigg|\\
&\leq 4n_1 \abs{G} \exp\Big(-\frac{\epsilon\delta \alpha n_1}{|G|^4}\Big)
+ 2\frac{1}{\abs{G}} \exp\big(-2\rho^2 n_1\big)
+2L_G \frac{1}{\abs{G}}  \Big(\exp(-c_{G,\delta}\alpha n_1) +2^{-\alpha n_1} \Big).
\end{align*}
The result follows.
\end{proof}

By setting $n_2=\ceil{\alpha n_1}$ and letting $n_1\to\infty$, we obtain the following.

\begin{proposition}\label{Prop: Laplacian main term}
Suppose $\frac{1}{p}<\alpha\leq 1$, $0<\rho<\min(\frac{1}{p},\alpha-\frac{1}{p})$ and $0<\delta<\frac{1}{6\log_p(\abs{G})}$.
Denote $S_1=\SSO{n,\delta}{G\oplus\ZG}{G\oplus\ZG}$ and $S_2=\SSO{\ceil{\alpha n},\delta}{G\oplus\ZG}{G\oplus\ZG}$.
Then, we have
\[
\lim_{n\to\infty} \abs{G}^{n+\ceil{\alpha n}-1}
\sum_{F_1\in S_1}
\sum_{F_2\in S_2}
\PPP[(F_1,F_2)\circ \Mna=0 \text{ and }
\abs{\gamma(\Mna) -\tfrac{1}{p} n}\leq \rho n] 
= \frac{1}{\abs{G}}.
\]
\end{proposition}

\subsection{Error term}

We are now left with the case when $T_1$, $H_1$, $T_2$, $H_2$, $G\oplus\ZG$ are not all the same.
Since $H_1+H_2=G\oplus\ZG$, at least of the following must hold:
\begin{itemize}
    \item $T_1\subsetneq H_1$;
    \item $T_2\subsetneq H_2$;
    \item $T_1\neq T_2$.
\end{itemize}
We will actually split this further into the following cases:
\begin{itemize}
    \item $T_1\subsetneq H_1$;
    \item $T_2\subsetneq H_2$;
    \item $T_1\not\subseteq T_2$ and $T_2=H_2$;
    \item $T_2\not\subseteq T_1$, $T_2=H_2$ and $\alpha\neq 1$;
    \item $T_1=H_1$, $T_2=H_2$, $T_1\subsetneq T_2$ and $\alpha=1$.
\end{itemize}
For $t\in\{1,2\}$, denote
\[
e_t
= \begin{cases}
    1-\epsilon &\text{if } T_t\subsetneq H_t\\
    1 &\text{if } T_t= H_t.
\end{cases}
\]
We will consider $F_1\in \SSO{n_1,\delta}{H_1}{T_1}$ and $F_2\in \SSO{n_2,\delta}{H_2}{T_2}$.

\begin{proposition}\label{Prop: Laplacian error term}
Suppose $\frac{1}{p}<\alpha\leq 1$ and $0<\rho< \min(\alpha-\frac{1}{p},\frac{1}{p})$.
Consider $\pi_2$-full subgroups $T_1\subseteq H_1\subseteq G\oplus\ZG$ and $T_2\subseteq H_2\subseteq G\oplus\ZG$ that satisfy at least one of the following
\begin{itemize}
    \item $T_1\subsetneq H_1$;
    \item $T_2\subsetneq H_2$;
    \item $T_1\not\subseteq T_2$ and $T_2=H_2$;
    \item $T_2\not\subseteq T_1$, $T_2=H_2$ and $\alpha\neq 1$.
\end{itemize}
Assuming $\delta>0$ is sufficiently small in terms of these, we have
\[
\lim_{n\to\infty}
\abs{G}^{n+\ceil{\alpha n}-1}
\sum_{F_1\in \SSO{n,\delta}{H_1}{T_1}}
\sum_{F_2\in \SSO{\ceil{\alpha n},\delta}{H_2}{T_2}}
\PPP[(F_1,F_2)\circ \Mna=0 \text{ and } \abs{\gamma(\Mna)-\tfrac{1}{p}n} \leq \rho n]
=0.
\]
\end{proposition}
\begin{proof}
Denote $D_1=[H_1:T_1]$, $D_2=[H_2:T_2]$, $n_1=n$ and $n_2=\ceil{\alpha n}$.
First, consider a fixed $S\subseteq[n_1]$. 
By Lemma~\ref{Lem: depth prob intermediate} we know that
\begin{align*}
&\sum_{F_1\in \SSO{n_1,\delta}{H_1}{T_1}}
\sum_{F_2\in \SSO{n_2,\delta}{H_2}{T_2}}
\PPP[(F_1,F_2)\circ \Mna=0 \text{ and } \Gamma(\Mna)=S]\\
&\leq (\tfrac{1}{p})^{\abs{S}} (\tfrac{1}{\abs{G}})^{n_1-\abs{S}}
\frac{e_1^{n_2} e_2^{\abs{S}}}{\abs{T_1}^{n_2} \abs{T_2}^{\abs{S}}}
\left(1+4\abs{G}^2n_1 \exp\Big(-\frac{\epsilon \delta \alpha n_1}{|G|^4}\Big) \right)\\
&\quad\quad\times 
\sum_{F_2\in \SSO{n_2,\delta}{H_2}{T_2}}
\sum_{F_1\in \SSO{n_1,\delta}{H_1}{T_1}}
\prod_{i\in [n_1]\setminus S} \sum_{l\in (\Z/\abs{G}\Z)^{\times}} \PPP[l F_2(B_i)= F_1(u_i)].
\end{align*}
We will fix a choice of $F_2$, and bound the sum
\[
\sum_{F_1\in \SSO{n_1,\delta}{H_1}{T_1}}
\prod_{i\in [n_1]\setminus S} \sum_{l\in (\Z/\abs{G}\Z)^{\times}} \PPP[l F_2(B_i)= F_1(u_i)].
\]

We write $V=\Zp^{n_1}$.
For a subset $\sigma\subseteq[n_1]$, denote by $V_{\sigma}$ the subgroup of $V$ generated by $u_i$ for $i\in \sigma$ and $V_{\setminus \sigma}$ to be the subgroup of $V$ generated by $u_i$ for $i\notin \sigma$.
Each $F_1\in \SSO{n_1,\delta}{H_1}{T_1}$, has some $\sigma\subseteq[n_1]$ with $\abs{\sigma}= \delta\log_p(D_1) n_1$ such that $F_1(V_{\setminus \sigma})\subseteq T_1$ and $F_1(V_{\sigma})\subseteq H_1$.

We fix a choice of $\sigma\subseteq [n_1]$ (in addition to fixing $F_2$ and $S$).
\begin{itemize}
    \item For $i\in S\cap \sigma$, since $H_1$ is $\pi_2$-full the number of choices for $F_1(u_i)$ is
    \[
    \abs{H_1\cap \pi_2^{-1}(1)}
    =\frac{\abs{H_1}}{\abs{G}}.
    \]
    \item For $i\in S\setminus \sigma$, since $T_1$ is $\pi_2$-full the number of choices for $F_1(u_i)$ is
    \[
    \abs{T_1\cap \pi_2^{-1}(1)}
    =\frac{\abs{T_1}}{\abs{G}}.
    \]
    \item For $i\in \sigma\setminus S$, $F_1(u_i)\in H_1\cap \pi_2^{-1}(1)$, 
    \begin{align*}
    \sum_{h\in H_1\cap \pi_2^{-1}(1)}
    \sum_{l\in (\Z/\abs{G}\Z)^{\times}} \PPP[l F_2(B_i)= h]
    &=\sum_{k\in (\Z/\abs{G}\Z)^{\times}}
    \sum_{h\in H_1\cap \pi_2^{-1}(1)}
    \PPP[F_2(B_i)= kh]\\
    &\leq \sum_{k\in (\Z/\abs{G}\Z)^{\times}}
    \PPP[\pi_2(F_2(B_i))= k]\\
    &=\PPP[\pi_2(F_2(B_i))\in (\ZG)^{\times}].
    \end{align*}
    By Corollary~\ref{Cor: prob pi_2 is unit}, this is at most
    \[
     \left( (1-\tfrac{1}{p})
     +\abs{G}^2\exp\Big(-\frac{\epsilon\delta n_2}{|G|^4}\Big) \right).
    \]

    \item For $i\in (S\cup \sigma)^c$, $F_1(u_i)\in T_1\cap \pi_2^{-1}(1)$, so
    \begin{align*}
    \sum_{h\in T_1\cap \pi_2^{-1}(1)}
    \sum_{l\in (\Z/\abs{G}\Z)^{\times}} \PPP[l F_2(B_i)= h]
    &=\sum_{l\in (\Z/\abs{G}\Z)^{\times}} \PPP[l F_2(B_i)\in T_1\cap \pi_2^{-1}(1)]\\
    &=\PPP[ F_2(B_i)\in T_1 \text{ and }\pi_2(F_2(B_i))\in (\ZG)^{\times}].
    \end{align*}
    Denote
    \[
    A=\begin{cases}
        \abs{T_1\cap T_2} &\text{if } T_2=H_2\\
        \min(\abs{T_1},\abs{T_2}) &\text{if } T_2\subsetneq H_2.
    \end{cases}
    \]
    By Lemma~\ref{Lem: bound subgroup, pi_2 full, code} and Lemma~\ref{Lem: bound subgroup, pi_2 full, depth}, this is at most
    \[
     \left( (1-\tfrac{1}{p}) \frac{A}{\abs{T_2}}
    +\abs{G}^2\exp\Big(-\frac{\epsilon\delta n_2}{|G|^4}\Big) \right).
    \]
\end{itemize}
We conclude that the sum over those $F_1$ corresponding to $\sigma$ is bounded by
\begin{align*}
&\sum_{F_1:\sigma}
\prod_{i\in [n_1]\setminus S} \sum_{l\in (\Z/\abs{G}\Z)^{\times}} \PPP[l F_2(B_i)= F_1(u_i)]\\
&\leq \left(\frac{\abs{H_1}}{\abs{G}}\right)^{\abs{S\cap \sigma}}
\left(\frac{\abs{T_1}}{\abs{G}}\right)^{\abs{S\setminus \sigma}}
\left( (1-\tfrac{1}{p})
     +\abs{G}^2\exp\Big(-\frac{\epsilon\delta\alpha n_1}{|G|^4}\Big) \right)^{\abs{\sigma\setminus S}}\\ 
&\quad\quad\times
\left( (1-\tfrac{1}{p}) \frac{A}{\abs{T_2}}
    +\abs{G}^2\exp\Big(-\frac{\epsilon\delta \alpha n_1}{|G|^4}\Big) \right)^{n_1-\abs{S\cup \sigma}}.
\end{align*}
Once $n_1$ is large enough to ensure that
\[
\frac{\abs{G}^2\exp\Big(-\frac{\epsilon\delta \alpha n_1}{|G|^4}\Big)}{2^{\frac{1}{n_1}}-1}
\leq (1-\tfrac{1}{p}) \frac{1}{\abs{G}^2},
\]
then by Corollary~\ref{Cor: estimate product}, the expression above is bounded by
\[
\left(\frac{\abs{H_1}}{\abs{G}}\right)^{\abs{S\cap \sigma}}
\left(\frac{\abs{T_1}}{\abs{G}}\right)^{\abs{S\setminus \sigma}}
(1-\tfrac{1}{p})^{n_1-\abs{S}} 
\left(\frac{A}{\abs{T_2}}\right)^{n_1-\abs{S\cup \sigma}}
\left(1+ 2n_1\frac{\abs{G}^2}{1-\frac{1}{p}} \abs{G}^2\exp\Big(-\frac{\epsilon\delta \alpha n_1}{|G|^4}\Big)\right).
\]
Note that
\[
\left(\frac{\abs{H_1}}{\abs{G}}\right)^{\abs{S\cap \sigma}}
\left(\frac{\abs{T_1}}{\abs{G}}\right)^{\abs{S\setminus \sigma}}
\leq \left(\frac{\abs{T_1}}{\abs{G}}\right)^{\abs{S}} D_1^{\abs{\sigma}},
\]
$\left(\frac{A}{\abs{T_2}}\right)^{n_1-\abs{S\cup \sigma}}\leq \left(\frac{A}{\abs{T_2}}\right)^{n_1-\abs{S}-\abs{\sigma}}$ and for sufficiently large $n_1$
\[
\left(1+ 2n_1\frac{\abs{G}^2}{1-\frac{1}{p}} \abs{G}^2\exp\Big(-\frac{\epsilon\delta \alpha n_1}{|G|^4}\Big)\right)
\leq 2.
\]
Therefore,
\begin{align*}
\sum_{F_1:\sigma}
\prod_{i\in [n_1]\setminus S} \sum_{l\in (\Z/\abs{G}\Z)^{\times}} \PPP[l F_2(B_i)= F_1(u_i)]
&\leq 2 (1-\tfrac{1}{p})^{n_1-\abs{S}}
\left(\frac{\abs{T_1}}{\abs{G}}\right)^{\abs{S}}
\left(\frac{A}{\abs{T_2}}\right)^{n_1-\abs{S}}
\left(\frac{D_1\abs{T_2}}{A}\right)^{\abs{\sigma}}.
\end{align*}
By summing over $\sigma\subseteq[n_1]$ with $\abs{\sigma}= \delta\log_p(D_1) n_1$, we see that
\begin{align*}
&\sum_{F_1\in \SSO{n_1,\delta}{H_1}{T_1}}
\prod_{i\in [n_1]\setminus S} \sum_{l\in (\Z/\abs{G}\Z)^{\times}} \PPP[l F_2(B_i)= F_1(u_i)]\\
&\leq \binom{n_1}{\delta\log_p(D_1) n_1}
2 (1-\tfrac{1}{p})^{n_1-\abs{S}}
\left(\frac{\abs{T_1}}{\abs{G}}\right)^{\abs{S}}
\left(\frac{A}{\abs{T_2}}\right)^{n_1-\abs{S}}
\left(\frac{D_1\abs{T_2}}{A}\right)^{\delta\log_p(D_1) n_1}.
\end{align*}
This implies that
\begin{align*}
&\sum_{F_1\in \SSO{n_1,\delta}{H_1}{T_1}}
\sum_{F_2\in \SSO{n_2,\delta}{H_2}{T_2}}
\PPP[(F_1,F_2)\circ \Mna=0 \text{ and } \Gamma(\Mna)=S]\\
&\leq (\tfrac{1}{p})^{\abs{S}} (\tfrac{1}{\abs{G}})^{n_1-\abs{S}}
\frac{e_1^{n_2} e_2^{\abs{S}}}{\abs{T_1}^{n_2} \abs{T_2}^{\abs{S}}}
\left(1+4\abs{G}^2n_1 \exp\Big(-\frac{\epsilon \delta \alpha n_1}{|G|^4}\Big) \right)\\
&\quad\quad\times \binom{n_2}{\delta n_2\log_p(D_2)} \frac{\abs{T_2}^{n_2}}{\abs{G}^{n_2}} D_2^{\delta n_2\log_p(D_2)}\\
&\quad\quad\times\binom{n_1}{\delta\log_p(D_1) n_1}
2 (1-\tfrac{1}{p})^{n_1-\abs{S}}
\left(\frac{\abs{T_1}}{\abs{G}}\right)^{\abs{S}}
\left(\frac{A}{\abs{T_2}}\right)^{n_1-\abs{S}}
\left(\frac{D_1\abs{T_2}}{A}\right)^{\delta\log_p(D_1) n_1}.
\end{align*}
Rearranging the terms and using the fact that $\binom{a}{b}\leq 2^{2\sqrt{ab}}$, we see that the above expression is at most
\begin{align*}
&2 (\tfrac{1}{\abs{G}})^{n_1+n_2}
(\tfrac{1}{p})^{\abs{S}}
 (1-\tfrac{1}{p})^{n_1-\abs{S}}
 e_1^{n_2} e_2^{\abs{S}}
 \left(\frac{A}{\abs{T_1}}\right)^{n_2-\abs{S}}
  \left(\frac{A}{\abs{T_2}}\right)^{n_1-n_2}\\
&\times 2^{2n_2\sqrt{\delta\log_p(D_2)}}
2^{2n_1\sqrt{\delta\log_p(D_1)}}
D_2^{\delta n_2\log_p(D_2)}\left(\frac{D_1\abs{T_2}}{A}\right)^{\delta\log_p(D_1) n_1}
\left(1+4\abs{G}^2n_1 \exp\Big(-\frac{\epsilon \delta \alpha n_1}{|G|^4}\Big) \right).
\end{align*}
Since $e_1,e_2,\frac{A}{\abs{T_1}}\frac{A}{\abs{T_2}}\leq 1$, $n_2=\ceil{\alpha n}$ and $(\frac{1}{p}-\rho)n \leq \abs{S}\leq (\frac{1}{p}+\rho)n$, we see that
\[
 e_1^{n_2} e_2^{\abs{S}}
 \left(\frac{A}{\abs{T_1}}\right)^{n_2-\abs{S}}
  \left(\frac{A}{\abs{T_2}}\right)^{n_1-n_2}
  \leq\left(e_1^{\alpha} e_2^{\frac{1}{p}-\rho} \left(\frac{A}{\abs{T_1}}\right)^{\alpha-\frac{1}{p}-\rho} \left(\frac{A}{\abs{T_2}}\right)^{1-\alpha} \right)^n \frac{\abs{T_2}}{A}.
\]
Next, note that
\begin{align*}
&2^{2n_2\sqrt{\delta\log_p(D_2)}}
2^{2n_1\sqrt{\delta\log_p(D_1)}}
D_2^{\delta n_2\log_p(D_2)}
\left(\frac{D_1\abs{T_2}}{A}\right)^{\delta\log_p(D_1) n_1}\\
&\leq 
\left(
2^{2\sqrt{2\log_p(\abs{G})}}
2^{2\sqrt{2\log_p(\abs{G})}}
\abs{G}^{4\log_p(\abs{G})}
\abs{G}^{8\log_p(\abs{G})}
\right)^{n\sqrt{\delta}}.
\end{align*}
Further note that for sufficiently large $n$, we have
\[
\left(1+4n_1\abs{G}^2\exp\Big(-\frac{\epsilon\delta \alpha n_1}{|G|^4}\Big)\right)
<2.
\]
We write
\[
B=2^{4(2\log_p(\abs{G}))^{1/2}}
\abs{G}^{12\log_p(\abs{G})}.
\]
Then we see that
\begin{align*}
&\sum_{F_1\in \SSO{n,\delta}{H_1}{T_1}}
\sum_{F_2\in \SSO{\ceil{\alpha n},\delta}{H_2}{T_2}}
\PPP[(F_1,F_2)\circ \Mna=0 \text{ and } \Gamma(\Mna)=S]\\
&\leq 4\abs{G}^2
\frac{1}{\abs{G}^{n+\ceil{\alpha n}}}
(\tfrac{1}{p})^{\abs{S}}  
(1-\tfrac{1}{p})^{n-\abs{S}}
\left(
e_1^{\alpha} e_2^{\frac{1}{p}-\rho}
\left(\frac{A}{\abs{T_1}}\right)^{\alpha-\frac{1}{p}-\rho}
\left(\frac{A}{\abs{T_2}}\right)^{1-\alpha}
B^{\sqrt{\delta}}
\right)^n.
\end{align*}
By summing over different $S$, we see that
\begin{align*}
&\sum_{F_1\in \SSO{n,\delta}{H_1}{T_1}}
\sum_{F_2\in \SSO{\ceil{\alpha n},\delta}{H_2}{T_2}}
\PPP[(F_1,F_2)\circ \Mna=0 \text{ and } \abs{\gamma(\Mna)-\tfrac{1}{p}n} \leq \rho n]\\
&\leq 4\abs{G}^2
\frac{1}{\abs{G}^{n+\ceil{\alpha n}}}
\left(
e_1^{\alpha} e_2^{\frac{1}{p}-\rho}
\left(\frac{A}{\abs{T_1}}\right)^{\alpha-\frac{1}{p}-\rho}
\left(\frac{A}{\abs{T_2}}\right)^{1-\alpha}
B^{\sqrt{\delta}}
\right)^n.
\end{align*}
Note that $e_1$, $e_2$, $\frac{A}{\abs{T_1}}$ and $\frac{A}{\abs{T_2}}$ are all at most one. Moreover, at least one of the following holds
\begin{itemize}
    \item $T_1\subsetneq H_1$.
    In this case $e_1<1$ and $\alpha>0$, so $e_1^\alpha<1$.
    \item $T_2\subsetneq H_2$.
    In this case $e_2<1$ and $\frac{1}{p}-\rho>0$, so $e_2^{\frac{1}{p}-\rho}<1$.
    \item $T_1\not\subseteq T_2$ and $T_2=H_2$.
    In this case $A=\abs{T_1\cap T_2}$,
    $\frac{\abs{T_1\cap T_2}}{\abs{T_1}}<1$ and $\alpha-\frac{1}{p}-\rho>0$, so $\Big(\frac{A}{\abs{T_1}}\Big)^{\alpha-\frac{1}{p}-\rho}<1$.
    \item $T_2\not\subseteq T_1$, $T_2=H_2$ and $\alpha\neq 1$.
    In this case $A=\abs{T_1\cap T_2}$, $\frac{\abs{T_1\cap T_2}}{\abs{T_2}}<1$ and $1-\alpha>0$, so $\Big(\frac{A}{\abs{T_2}}\Big)^{1-\alpha}<1$.
\end{itemize}
Therefore,
\[
e_1^{\alpha} e_2^{\frac{1}{p}-\rho}
\left(\frac{A}{\abs{T_1}}\right)^{\alpha-\frac{1}{p}-\rho}
\left(\frac{A}{\abs{T_2}}\right)^{1-\alpha}<1.
\]
We can therefore choose a sufficiently small $\delta$ to ensure that
\[
e_1^{\alpha} e_2^{\frac{1}{p}-\rho}
\left(\frac{A}{\abs{T_1}}\right)^{\alpha-\frac{1}{p}-\rho}
\left(\frac{A}{\abs{T_2}}\right)^{1-\alpha}
B^{\sqrt{\delta}}
<1.
\]
The result follows.
\end{proof}

Now we are only left with the case $T_1=H_1$, $T_2=H_2$, $T_1\subsetneq T_2$ and $\alpha=1$.

\begin{proposition}\label{Prop: Laplacian error term alpha 1}
Suppose $\alpha= 1$ and $0<\rho< \min(\alpha-\frac{1}{p},\frac{1}{p})$.
Consider $\pi_2$-full subgroups $T_1 \subsetneq T_2 \subseteq G\oplus\ZG$.
For $\delta>0$, we have
\[
\lim_{n\to\infty}
\abs{G}^{2n-1}
\sum_{F_1\in \SSO{n,\delta}{T_1}{T_1}}
\sum_{F_2\in \SSO{n,\delta}{T_2}{T_2}}
\PPP[(F_1,F_2)\circ \Mna=0 \text{ and } \abs{\gamma(\Mna)-\tfrac{1}{p}n} \leq \rho n]
=0.
\]
\end{proposition}
\begin{proof}
Since $\alpha=1$, $n_1=n_2=n$.
We have $e_1=e_2=1$, $F_1\in \SSO{n,\delta}{T_1}{T_1}$ are codes of distance $\delta n$ with image $T_1$ and $F_2\in \SSO{n,\delta}{T_2}{T_2}$ are codes of distance $\delta n$ with image $T_2$.

By the independence of columns of $\Mnot$, we have
\[
\PPP[(F_1,F_2)\circ \Mnot=0]
=\prod_{i\in [n_1]} \PPP[F_2(B_i)=-c_iF_1(u_i)]
\prod_{j\in [n_2]} \PPP[F_1(A_j)=-d_jF_2(v_j)].
\]
For each $i\in [n]$, we know by Corollary~\ref{Cor: combined bound} that
\[
\PPP[F_2(B_i)=-c_iF_1(u_i)]
\leq  \frac{1}{\abs{T_2}} + \exp\Big(-\frac{\epsilon \delta n}{|G|^4}\Big).
\]
For each $j\in [n]$, we know from Corollary~\ref{Cor: combined bound} that
\[
\PPP[F_1(A_j)= -d_jF_2(v_j)]
\leq \frac{1}{\abs{T_1}} + \exp\Big(-\frac{\epsilon \delta n}{|G|^4}\Big).
\]
However, if $F_2(v_j)\in T_2\setminus T_1$, then we can get a tighter upper bound. Since $F_1(A_j)$ can only take values in $T_1$, if $F_1(A_j)= -d_jF_2(v_j)$ then $d_jF_2(v_j)\in T_1$. If $F_2(v_j)\notin T_1$, this forces $d_j\equiv 0\pmod{p}$.
We conclude that if $F_2(v_j)\in T_2\setminus T_1$, then
\begin{align*}
\PPP[F_1(A_j)= -d_jF_2(v_j)]
&= \PPP[F_1(A_j)= -d_jF_2(v_j) \mid d_jF_2(v_j)\in T_1]
\PPP[d_jF_2(v_j)\in T_1]\\
&\leq \left( \frac{1}{\abs{T_1}} + \exp\Big(-\frac{\epsilon \delta n}{|G|^4}\Big)\right)
\PPP[d_j\equiv 0\pmod{p}]\\
&= \frac{1}{p}
\left( \frac{1}{\abs{T_1}} + \exp\Big(-\frac{\epsilon \delta n}{|G|^4}\Big)\right).
\end{align*}
Since $F_2$ is a code of distance $\delta n$ with image $T_2$, we know that 
\[
\abs{\{j\in [n]: F_2(v_j)\notin T_1\}}
\geq \delta n.
\]
It follows that
\[
\PPP[(F_1,F_2)\circ \Mnot=0 ]
\leq 
\left( \frac{1}{\abs{T_2}} + \exp\Big(-\frac{\epsilon \delta n}{|G|^4}\Big)\right)^{n}
\Big(\frac{1}{p}\Big)^{\delta n}
\left(\frac{1}{\abs{T_1}} + \exp\Big(-\frac{\epsilon \delta n}{|G|^4}\Big)\right)^{n}.
\]
Once $n$ is large enough to ensure,
\[
\frac{\exp\Big(-\frac{\epsilon\delta  n}{|G|^4}\Big)}{2^{\frac{1}{2n}}-1}
<\frac{1}{\abs{G}^2},
\]
then by Corollary~\ref{Cor: estimate product} we know that
\[
\PPP[(F_1,F_2)\circ \Mna=0]
\leq \Big(\frac{1}{p}\Big)^{\delta n}
\frac{1}{\abs{T_2}^n \abs{T_1}^n}
\left( 1 + 4n\abs{G}^2\exp\Big(-\frac{\epsilon \delta n}{|G|^4}\Big)\right).
\]
Since $\abs{\SSO{n,\delta}{T_1}{T_1}}\leq \frac{\abs{T_1}^n}{\abs{G}^n} $ and $\abs{\SSO{n,\delta}{T_2}{T_2}}\leq \frac{\abs{T_2}^n}{\abs{G}^n}$, we conclude that
\begin{align*}
&\sum_{F_1\in \SSO{n,\delta}{T_1}{T_1}}
\sum_{F_2\in \SSO{n,\delta}{T_2}{T_2}}
\PPP[(F_1,F_2)\circ \Mna=0 \text{ and } \abs{\gamma(\Mna)-\tfrac{1}{p}n} \leq \rho n]\\
&\leq\sum_{F_1\in \SSO{n,\delta}{T_1}{T_1}}
\sum_{F_2\in \SSO{n,\delta}{T_2}{T_2}}
\PPP[(F_1,F_2)\circ \Mna=0]\\
&\leq \frac{1}{\abs{G}^{2n}} \Big(\frac{1}{p}\Big)^{\delta n}
\left( 1 + 4n\abs{G}^2\exp\Big(-\frac{\epsilon \delta n}{|G|^4}\Big)\right).
\end{align*}
The result follows.
\end{proof}

Notice that Proposition~\ref{Prop: Laplacian main term}, Proposition~\ref{Prop: Laplacian error term} and Proposition~\ref{Prop: Laplacian error term alpha 1} together prove Proposition~\ref{Prop: Laplacian pieces}. We have previously seen that Proposition~\ref{Prop: Laplacian pieces} implies Proposition~\ref{Prop: Sur moment Laplacian}, which implies Theorem~\ref{Thm: distribution Laplacian}, which in turn implies Theorem~\ref{Thm: dist directed bipartite graph}.

\section{Raw moments sometimes diverge to infinity}\label{Sec: Raw moment infty}

We consider the special case of $\Mna$ when all $a_{ij}$ and $b_{ij}$ are Haar-uniform and show that certain moments blow up to infinity.

\begin{proposition}
Consider $0<\alpha<1$ and $m>\frac{2-\alpha}{1-\alpha}$.
If all $a_{ij}$ and $b_{ij}$ in $\Mna$ are Haar-uniform, then we have
\[
\lim_{n\to\infty} 
\E[\abs{\Hom(\Coker(\Mna),(\Z/p\Z)^m)}]
=\infty.
\]
\end{proposition}
\begin{proof}
Denote $G=(\Z/p\Z)^m$, $n_1=n$ and $n_2=\ceil{\alpha n}$. We know that
\[
\E[\abs{\Hom(\Coker(\Mna),(\Z/p\Z)^m)}]
=\sum_{F_1\in\Hom(\Zp^{n_1},G)}
\sum_{F_2\in\Hom(\Zp^{n_2},G)}
\PPP[(F_1,F_2)\circ \Mna =0].
\]
Since the columns of $\Mna$ are independent, we see that
\[
\PPP[(F_1,F_2)\circ \Mna =0]
=\prod_{i=1}^{n_1} \PPP[F_2(B_i)=-c_iF_1(u_i)]
\prod_{j=1}^{n_2} \PPP[F_1(A_j)=-d_jF_2(v_j)].
\]

Let $H$ be a subgroup of $G$ isomorphic to $\Z/p\Z$. As a lower bound, we restrict ourselves to $F_1\in\Sur(\Zp^{n_1},G)$ and $F_2\in \Sur(\Zp^{n_2},H)$.

By Lemma~\ref{Lem: Haar uniform}, we know that
    \[
    \PPP[F_1(A_j)=-d_jF_2(v_j)]
    =\frac{1}{\abs{G}}=\frac{1}{p^m}.
    \]
\begin{itemize}
    \item If $F_1(u_i)\in H$, then by Lemma~\ref{Lem: Haar uniform}, we know that
    \[
    \PPP[F_2(B_i)=-c_iF_1(u_i)]
    =\frac{1}{p}.
    \]
    \item If $F_1(u_i)\notin H$, then 
    \[
    \PPP[c_i F_1(u_i)\in H]
    =\PPP[c_i \equiv0\pmod{p}]
    =\frac{1}{p}.
    \]
    Therefore, by Lemma~\ref{Lem: Haar uniform}, we know that
    \[
    \PPP[F_2(B_i)=-c_iF_1(u_i)]
    =\PPP[F_2(B_i)=-c_iF_1(u_i)\mid c_iF_1(u_i)\in H] \PPP[c_iF_1(u_i)\in H]
    =\frac{1}{p^2}.
    \]
\end{itemize}
We see that regardless of whether $F_1(u_i)$ is in $H$ or not we have
\[
\PPP[F_2(B_i)=-c_iF_1(u_i)]
\geq \frac{1}{p^2}.
\]
We conclude that
\[
\PPP[(F_1,F_2)\circ \Mna =0]
\geq p^{-m \ceil{\alpha n} -2 n}
\geq p^{- (m\alpha+2)n -m}.
\]
Since $\abs{\Sur(\Zp^{n_1},G)} =p^{mn}(1+o(1))$ and $\abs{\Sur(\Zp^{n_2},H)} =p^{\alpha n}(1+o(1))$, we conclude that
\begin{align*}
\E[\abs{\Hom(\Coker(\Mna),(\Z/p\Z)^m)}]
&\geq p^{- (m\alpha+2)n -m} p^{(m+\alpha)n} (1+o(1))\\
&= p^{-m} (1+o(1)) \big(p^{m(1-\alpha)+\alpha-2} \big)^n.
\end{align*}
Since $m>\frac{2-\alpha}{1-\alpha}$, this goes to infinity as $n\to\infty$.
\end{proof}

\section{Technical Lemmas}\label{Sec: tech lemma dir}

\begin{lemma}\label{Lem: bound tail}
Given $\rho>0$, we have
\[
\bigg| 1-\sum_{k:\abs{k-\frac{1}{p}n}\leq \rho n}\binom{n}{k} \frac{1}{p^k} \Big(1-\frac{1}{p}\Big)^{n-k} \bigg|
\leq 2\exp\big(-2\rho^2 n\big).
\]
\end{lemma}
\begin{proof}
The result follows from Hoeffding's inequality.
\end{proof}

\begin{corollary}\label{Cor: prob of cond to 1}
Given $\rho>0$, we have
\[
\lim_{n\to\infty} 
\PPP[\abs{\gamma(\Mna)- \tfrac{1}{p} n} \leq \rho n] 
=1.
\]
\end{corollary}

\begin{lemma}\label{Lem: gamma Lna concentrate}
Given $\rho>0$, we have
\[
\lim_{n\to\infty}
\PPP[|\gamma(\Lna)-\tfrac{1}{p} n|\leq \rho n]
=1.
\]
\end{lemma}
\begin{proof}
Denote $n_2=\ceil{\alpha n}$.
For each $1\leq i\leq n$, we denote
\[
\beta_i
=\begin{cases}
    1 &\text{ if } \sum_{j=1}^{n_2} b_{ij}\equiv 0\pmod{p}\\
    0 &\text{ if } \sum_{j=1}^{n_2} b_{ij}\not\equiv 0\pmod{p}.
\end{cases}
\]
Therefore $\gamma(\Lna)=\sum_{i=1}^{n} \beta_i$.

Consider the map $F_0$ from $\Zp^{n_2}$ to $\Z/p\Z$ that sends each basis vector to $1$. So $\sum_{j} b_{ij}\equiv 0\pmod{p}$ if and only if $F_0(B_i)=0$. Note that $F_0$ is a code of distance $n_2$ with image $\Z/p\Z$.
Lemma~\ref{Lem: Wood estimate code} tells us that for each $i$,
\[
\left| \PPP[F_0(B_i)=0]-\frac{1}{p}\right|
\leq \exp\Big(-\frac{\epsilon n_2}{p^2}\Big)
\leq \exp\Big(-\frac{\epsilon \alpha n}{p^2}\Big).
\]
Now,
\[
\E[\gamma(\Lna)]
=\sum_{i=1}^{n} \E[\beta_i]
=\sum_{i=1}^{n} \PPP[F_0(B_i)=0],
\]
and hence
\[
\left| \E[\gamma(\Lna)] 
-\frac{1}{p} n\right|
\leq n \exp\Big(-\frac{\epsilon \alpha n}{p^2}\Big).
\]
Next,
\[
\Var(\beta_i)
= \PPP[F_0(B_i)=0] (1- \PPP[F_0(B_i)=0])
\leq 1.
\]
Moreover, since the $\beta_i$ are independent, it follows that
\[
\Var(\gamma(\Lna))
=\sum_{i=1}^{n} \Var(\beta_i)
\leq n.
\]

Notice that
\[
\E\Big[\big(\gamma(\Lna)-\tfrac{1}{p}n\big)^2\Big]
=\Var(\gamma(\Lna))
+\big(\E[\gamma(\Lna)] -\tfrac{1}{p}n\big)^2
\leq n + n^2 \exp\Big(-\frac{2\epsilon \alpha n}{p^2}\Big).
\]
So by Chebyshev's inequality, we have
\[
\PPP[|\gamma(\Lna)-\tfrac{1}{p} n|\geq \rho n]
\leq \frac{\E\big[\big(\gamma(\Lna)-\tfrac{1}{p}n\big)^2\big]}{(\rho n)^2}
\leq \frac{n + n^2 \exp\Big(-\frac{2\epsilon \alpha n}{p^2}\Big)}{\rho^2 n^2}.
\]
This implies that
\[
\lim_{n\to\infty} \PPP[|\gamma(\Lna)-\tfrac{1}{p} n|\geq \rho n]
=0.\qedhere
\]
\end{proof}

\begin{lemma}\label{Lem: bound product}
Suppose we have $n\geq 2$, $0<k_0\leq x_1,x_2,\dots,x_n$ and $0<y<k_0(2^{1/n}-1)$, then we have
\[ 
\left(1-\frac{2ny}{k_0} \right) 
\prod_{i=1}^n x_i   
\leq \prod_{i=1}^n(x_i-y)
\leq \prod_{i=1}^n(x_i+y)
\leq  \left(1+\frac{2ny}{k_0}\right) 
\prod_{i=1}^n x_i.
\]
\end{lemma}
\begin{proof}
Consider $f_1(t)=\prod_{i=1}^n(x_i+t)$,
$g_1(t)=(1+\frac{2nt}{k_0}) \prod_{i=1}^n x_i$,
$f_2(t)=\prod_{i=1}^n(x_i-t)$ and
$g_2(t)=(1-\frac{2nt}{k_0}) \prod_{i=1}^n x_i$. Now we have $f_1(0)=g_1(0)$ and for $0\leq t\leq k_0(2^{1/n}-1)$, we have
\begin{align*}
f_1'(t)
&=\sum_{j=1}^{n}\prod_{\substack{1\leq i\leq n\\ i\neq j}} (x_i+t)
\leq \frac{n}{k_0} \prod_{1\leq i\leq n} (x_i+t)
= \frac{n}{k_0} \prod_{1\leq i\leq n} x_i \prod_{1\leq i\leq n} (1+\frac{t}{x_i} )\\
&\leq \frac{n}{k_0} \Big(\prod_{1\leq i\leq n} x_i \Big)  \Big(1+\frac{t}{k_0} \Big)^n
\leq \frac{2n}{k_0} \Big(\prod_{1\leq i\leq n} x_i \Big)  =g_1'(t).
\end{align*}
We can therefore conclude that $f_1(y)\leq g_1(y)$. Next, we have $f_2(0)=g_2(0)$ and for $0\leq t\leq k_0(2^{1/n}-1)$, we have
\begin{align*}
-f_2'(t)
&=\sum_{j=1}^{n}\prod_{\substack{1\leq i\leq n\\ i\neq j}} (x_i-t)
\leq \frac{n}{k_0-t} \prod_{1\leq i\leq n} (x_i-t)
\leq \frac{2n}{k_0} \prod_{1\leq i\leq n} x_i  =-g_2'(t).
\end{align*}
We conclude that $g_2(y)\leq f_2(y)$.
\end{proof}
\begin{corollary}\label{Cor: estimate product}
Suppose we have $n\geq 2$ and real numbers $k_0,x_1,\dots,x_n,x'_1,\dots,x'_n,y$ for which $0<k_0\leq x_1,x_2,\dots,x_n$, $0<y<k_0(2^{1/n}-1)$ and $|x_i'-x_i|\leq y$, then we have
\[\left|\prod_{i=1}^n x'_i -\prod_{i=1}^n x_i\right|\leq \frac{2ny}{k_0} \prod_{i=1}^n x_i.\]
\end{corollary}

\section*{Acknowledgments}
We thank Nathan Kaplan for introducing me to this problem and for many helpful discussions about the problem.
We acknowledge support from NSF Grant DMS 2154223.

\bibliographystyle{plain}
\bibliography{Bibliography}

\begin{thebibliography}{10}

\bibitem{bhargava2023rank}
Atal Bhargava, Jack DePascale, and Jake Koenig.
\newblock The rank of the sandpile group of random directed bipartite graphs.
\newblock {\em Annals of Combinatorics}, 27(4):979--992, 2023.

\bibitem{bhargava2015modeling}
Manjul Bhargava, Daniel~M Kane, Hendrik~W Lenstra, Bjorn Poonen, and Eric Rains.
\newblock Modeling the distribution of ranks, selmer groups, and shafarevich--tate groups of elliptic curves.
\newblock {\em Cambridge Journal of Mathematics}, 3(3):275--321, 2015.

\bibitem{cheong2025cokernel}
Gilyoung Cheong and Yifeng Huang.
\newblock The cokernel of a polynomial push-forward of a random integral matrix with concentrated residue.
\newblock In {\em Mathematical Proceedings of the Cambridge Philosophical Society}, volume 178, pages 229--257. Cambridge University Press, 2025.

\bibitem{cheong2022generalizations}
Gilyoung Cheong and Nathan Kaplan.
\newblock Generalizations of results of {F}riedman and {W}ashington on cokernels of random p-adic matrices.
\newblock {\em Journal of Algebra}, 604:636--663, 2022.

\bibitem{cheong2023distribution}
Gilyoung Cheong and Myungjun Yu.
\newblock The distribution of the cokernel of a polynomial evaluated at a random integral matrix.
\newblock {\em arXiv preprint arXiv:2303.09125}, 2023.

\bibitem{clancy2015cohen}
Julien Clancy, Nathan Kaplan, Timothy Leake, Sam Payne, and Melanie~Matchett Wood.
\newblock On a {C}ohen--{L}enstra heuristic for jacobians of random graphs.
\newblock {\em Journal of Algebraic Combinatorics}, 42(3):701--723, 2015.

\bibitem{friedman1989distribution}
Eduardo Friedman and Lawrence~C Washington.
\newblock On the distribution of divisor class groups of curves over a finite field.
\newblock In {\em Th{\'e}orie des nombres}, pages 227--239. de Gruyter Berlin, 1989.

\bibitem{FulmanKaplanSinghalWarnaar_SylowSandpileBipartite}
Jason Fulman, Nathan Kaplan, Deepesh Singhal, and Ole Warnaar.
\newblock Sandpile groups of random bipartite graphs and families of distributions with the same moments.
\newblock Preprint, 2026.

\bibitem{hodges2024distribution}
Eliot Hodges.
\newblock The distribution of sandpile groups of random graphs with their pairings.
\newblock {\em Transactions of the American Mathematical Society}, 377(12):8769--8815, 2024.

\bibitem{koplewitz2017sandpile}
Shaked Koplewitz.
\newblock Sandpile groups and the coeulerian property for random directed graphs.
\newblock {\em Advances in Applied Mathematics}, 90:145--159, 2017.

\bibitem{lee2023joint}
Jungin Lee.
\newblock Joint distribution of the cokernels of random p-adic matrices.
\newblock In {\em Forum Mathematicum}, volume~35, pages 1005--1020. De Gruyter, 2023.

\bibitem{meszaros2020distribution}
Andr{\'a}s M{\'e}sz{\'a}ros.
\newblock The distribution of sandpile groups of random regular graphs.
\newblock {\em Transactions of the American Mathematical Society}, 373(9):6529--6594, 2020.

\bibitem{nguyen2022random}
Hoi~H Nguyen and Melanie~Matchett Wood.
\newblock Random integral matrices: universality of surjectivity and the cokernel.
\newblock {\em Inventiones mathematicae}, 228(1):1--76, 2022.

\bibitem{nguyen2025local}
Hoi~H Nguyen and Melanie~Matchett Wood.
\newblock Local and global universality of random matrix cokernels.
\newblock {\em Mathematische Annalen}, 391(4):5117--5210, 2025.

\bibitem{shen2026quantative}
Jiahe Shen.
\newblock Quantative universality for cokernels of matrices with symmetries.
\newblock {\em arXiv preprint arXiv:2601.09704}, 2026.

\bibitem{Singhal_SandpileBipartite_Undirected}
Deepesh Singhal.
\newblock Distribution of sandpile groups of random bipartite graphs.
\newblock Preprint, 2026.

\bibitem{wood2017distribution}
Melanie Wood.
\newblock The distribution of sandpile groups of random graphs.
\newblock {\em Journal of the American Mathematical Society}, 30(4):915--958, 2017.

\bibitem{wood2019random}
Melanie~Matchett Wood.
\newblock Random integral matrices and the cohen-lenstra heuristics.
\newblock {\em American Journal of Mathematics}, 141(2):383--398, 2019.

\end{thebibliography}

\end{document}